\documentclass[psamsfont,12pt]{amsart}
\usepackage[usenames]{xcolor}
\usepackage{amsmath,amssymb}
\usepackage{dsfont}\let\mathbb\mathds
\usepackage{xypic}
\usepackage{amsthm,amsfonts,amsmath,amssymb,latexsym,epsfig,mathrsfs,yfonts,marvosym}

\usepackage{hyperref}
\usepackage[english]{babel} 

\newtheorem{theorem}{Theorem}[section]
\newtheorem{lemma}[theorem]{Lemma}
\newtheorem{corollary}[theorem]{Corollary}

\newtheorem{proposition}[theorem]{Proposition}

\newenvironment{example}{\medskip \refstepcounter{theorem}
\noindent  {\bf Example \thetheorem}.\rm}{\,}
\newtheorem*{ack}{Acknowledgements}

\theoremstyle{remark}
\newtheorem{rem}[theorem]{Remark}
\theoremstyle{definition}

\newcommand{\nts}[1]{\marginpar{#1}}
\renewcommand{\nts}[1]{}
\newcommand{\R}{\mathbf{R}}

  \def\bfV{\mbox{{\bf V}}} \def\bfS{\mbox{{\bf S}}}    \def\bfH{\mbox{{\bf H}}}      \def\bfF{\mbox{{\bf F}}}

        \def\mL{\mathcal{L}}               

   \def\bC{\mathbb C}        

\def\R{\mathbb R}\def\C{\mathbb C}

\def\BOne{{\mathchoice {\rm 1\mskip-4mu l} {\rm 1\mskip-4mu l}
                          {\rm 1\mskip-4.5mu l} {\rm 1\mskip-5mu l}}}

\def\fract#1#2{\raise4pt\hbox{$ #1 \atop #2 $}}

\def\bbc{{\mathbb C}}

\def\bbp{{\mathbb P}}
\def\bbq{{\mathbb Q}}
\def\bbr{{\mathbb R}}

\def\bbz{{\mathbb Z}}

\def\grl{\lambda}

\def\gro{\omega}

\def\grD{\Delta}

\def\grS{\Sigma}

\def\bfl{{\bf l}}

\def\bfv{{\bf v}}
\def\bfw{{\bf w}}

\def\cald{{\mathcal D}}

\def\calf{{\mathcal F}}
\def\calg{{\mathcal G}}

\def\la#1{\hbox to #1pc{\leftarrowfill}}
\def\ra#1{\hbox to #1pc{\rightarrowfill}}

\def\gp{{\mathfrak p}}

\def\gr{{\mathfrak r}}

\def\gt{{\mathfrak t}}

\def\gz{{\mathfrak z}}

\def\hook{\mathbin{\hbox to 6pt{%
                 \vrule height0.4pt width5pt depth0pt
                 \kern-.4pt
                 \vrule height6pt width0.4pt depth0pt\hss}}}

\usepackage{marvosym}

\def\del{\partial}              

\def\kt{\mathfrak{t}}

\def\ra{\rightarrow}

\def\Sas{\mbox{Sas}}

\begin{document}

\title{The Einstein--Hilbert functional and the Sasaki--Futaki invariant}
\date{\today}
\author{Charles P. Boyer}
\author{Hongnian Huang}
\author{Eveline Legendre}
\author{Christina W. T{\o}nnesen-Friedman}
\address{Charles P. Boyer, Department of Mathematics and Statistics,
University of New Mexico, Albuquerque, NM 87131.}
\email{cboyer@math.unm.edu} 
\address{Hongnian Huang, Department of Mathematics and Statistics,
University of New Mexico, Albuquerque, NM 87131.}
\email{hnhuang@unm.edu}
\address{Eveline Legendre\\ Universit\'e Paul Sabatier\\
Institut de Math\'ematiques de Toulouse\\ 118 route de Narbonne\\
31062 Toulouse\\ France}
\email{eveline.legendre@math.univ-toulouse.fr}
\address{Christina W. T{\o}nnesen-Friedman, Department of Mathematics, Union
College, Schenectady, New York 12308, USA } \email{tonnesec@union.edu}

\thanks{The first author was partially supported by grant \#245002 from the Simons Foundation. The third author is partially supported by France ANR project EMARKS No ANR-14-CE25-0010. The fourth author was partially supported by grant \#208799 from the Simons Foundation.}

%

\begin{abstract} We show that the Einstein--Hilbert functional, as a functional on the space of Reeb vector fields, detects the vanishing Sasaki-Futaki invariant. In particular, this provides an obstruction to the existence of a constant scalar curvature Sasakian metric. As an application we prove that K-semistable polarized Sasaki manifold has vanishing Sasaki-Futaki invariant. We then apply this result to show that under the right conditions on the Sasaki join manifolds of \cite{BoToJGA} a polarized Sasaki manifold is K-semistable if only if it has constant scalar curvature.

\end{abstract}

\maketitle

\markboth{The EH-functional and the SF--invariant}{C. Boyer, H. Huang, E. Legendre and C. T{\o}nnesen-Friedman}

\section{Introduction}

In the past ten years, Sasakian geometry has been in the middle of intense activities both in geometry and theoretical physics due to its role in the AdS/CFT correspondence~\cite{BG:book,reebMETRIC,MSYvolume}. It is, more precisely, a Sasaki--Einstein metric of positive scalar curvature that takes part in this matter, but finding obstructions or sufficient conditions for the existence of such a structure has led to an extensive exploration of Sasakian geometry~\cite{BG:book,FutakiOnoWang}. Via its transversal geometry, a Sasaki manifold involves a K\"ahler structure and, as such, the search for a Sasaki--Einstein metric is closely related to the search of a K\"ahler--Einstein metric on which there has been put a great deal of effort. See for e.g. \cite{aubin, CDSI, CDSII, CDSIII, futaki, T, T15, T15b, Y}. Sasaki--Einstein geometry is a very restrictive version of a constant scalar curvature Sasaki (cscS) metric, or even more generally an extremal Sasaki metric~\cite{BGS} and can be viewed as an odd dimensional analogue of the more classical 
subject of constant scalar curvature K\"ahler metrics, which has been actively studied since the pioneering works of Calabi~\cite{calabi}. Sasaki--Einstein metrics may occur when the first Chern class $c_1(D)$ of the contact distribution $D$ vanishes. Because of this cohomological constraint and because the K\"ahler--Einstein equation turns out to be a Monge--Amp\`ere equation, the cscS or analogously constant scalar curvature K\"ahler (cscK) problem, is even more difficult.  

One of the most famous obstructions to the existence of a cscK metric is the Futaki invariant~\cite{futaki}. In this paper we show that the Sasaki version of this invariant, namely the Sasaki-Futaki invariant or tranversal Futaki invariant, defined in~\cite{BGS,FutakiOnoWang}, is closely related to a modified version of the {\it Einstein--Hilbert functional}
 \begin{equation}\label{EHeqn} 
 \bfH(\xi) = \frac{\bfS^{n+1}_{\xi}}{\bfV_\xi^n}
 \end{equation} 
where $\bfS_\xi$ denotes the total transversal scalar curvature, $\bfV_\xi$ the volume, and $2n+1$ is the dimension of the Sasaki manifold. They are both defined as functionals on the cone of compatible Reeb vector fields, the {\it Sasaki cone}. The functional \eqref{EHeqn} is convenient since it is invariant under scaling of the Reeb vector field. The details are explained in~\S\ref{secBG} and \S\ref{secEH}. This functional is a slight modification of the original Einstein--Hilbert functional used in the resolution of the Yamabe problem~\cite{besse}.  

  \begin{theorem}\label{theoCRIT=zero} The set of critical points of the Einstein--Hilbert functional is the union of the zeros of the Sasaki-Futaki invariant and of the total transversal scalar curvature.  
  
 In particular, if a Reeb vector field admits a compatible cscS metric then it is a critical point of the Einstein--Hilbert functional. 
 \end{theorem}
More precisely, Lemma~\ref{lemmaCRIT} gives an explicit relation between the derivative of $\bfH$ and the Sasaki-Futaki invariant. 
 \begin{rem}
  The relation $Ric(\xi,\xi)=2n$ which holds for any Sasaki metric with Reeb vector field $\xi$ imposes that a Sasaki--Einstein metric has scalar curvature $2n$. Hence Sasaki--Einstein metric only lies in the transversal subset of the Sasaki cone on which $\bfS = 4n\bfV$. On that subset the Einstein--Hilbert functional is some constant, depending on $n$, times the volume functional which is convex, due to a result of Martelli--Sparks--Yau~\cite{reebMETRIC}.        
 \end{rem}

 Theorem~\ref{theoCRIT=zero} can be interpreted as an extension of a result of Martelli--Sparks--Yau~\cite{reebMETRIC,MSYvolume} to the cscS problem. Theorem~\ref{theoCRIT=zero} also generalizes a result of the third author~\cite{moi2} to the non toric case. Note that, contrary to the Sasaki--Einstein case, there is no chance to prove that
 $\bfH$ is (transversally) convex since, as first shown in ~\cite{moi2}, there are examples of multiple non-isometric cscS metrics in a given Sasaki cone (see also \cite{BoToJGA}). Indeed, here we give other explicit examples of this lack of non-uniqueness.
 
 One important asset of the Einstein--Hilbert functional is that it is much easier to compute than the Sasaki-Futaki invariant. Indeed,  it allows us to give an explicit expression for the Sasaki--Futaki invariant as a rational function when applied to the $\bfw$-cone of the weighted $S^3$-join manifolds $M_{l_1,l_2,\bfw}$ studied in \cite{BoToJGA}, see Lemma \ref{explicitSF} below.  Moreover, in the toric case, the total transversal scalar curvature is expressed as an integral of a certain polytope and it turns out that the Einstein--Hilbert functional coincides with the isoperimetric functional on polytopes tranverse to the moment cone~\cite{moi2}. In that case too, $\bfH$ is a rational function.

 Another obstruction to cscS metrics is the lack of K-semistability, defined in the Sasakian context as the K-semistability of the associated polarized cone $(Y,\xi)$ by Collins and Sz\'ekelyhidi in~\cite{CollinsSz}. By analogy with the K\"ahler case and in the light of the Donaldson--Tian--Yau conjecture~\cite{don:csc,tian:conj,yau} it is natural to wonder if this condition is also sufficient. We apply Theorem~\ref{theoCRIT=zero} toward an affirmative answer by proving the following theorem.

\begin{theorem}\label{theoSTABcrit}
If $(Y,\xi)$ is K-semistable then the Sasaki--Futaki invariant $\bfF_\xi$ vanishes identically, and $\xi$ is a critical point of the Einstein-Hilbert functional $\bfH(\xi)$. Alternatively, if $\xi$ is not a critical point of $\bfH(\xi)$ then $(Y,\xi)$ is not K-semistable. In particular, if $\xi$ is not a critical point of the Einstein-Hilbert functional then $(Y,\xi)$ is K-unstable.   
\end{theorem}

Note that Donaldson's proof \cite{don:estimate,don:extMcond,don:csc} of the Donaldson--Tian--Yau conjecture for compact toric surfaces readily implies that the Sasakian K-semistability of Collins--Sz\'ekelyhidi ensures the existence of a compatible cscS metric. Indeed, every transversal geometrical object is translated as an object defined on a transversal labelled polytope \cite{contactNOTE,moi2} and the K--stability is also defined only using the labelled polytope data~\cite{don:csc}.  \\   

As suggested by an example constructed in~\cite{don:csc}, and contrary to the toric Sasaki-Einstein problem, it is generally unlikely that a Reeb vector field with vanishing Sasaki--Futaki invariant necessarily admits a compatible cscS metric. However, when all elements of the Sasaki cone can be represented by extremal Sasaki metrics, we have

\begin{theorem}\label{critcscS}
Suppose the Sasaki cone is exhausted by extremal Sasaki metrics and that the total transverse scalar curvature does not vanish. Then the set of critical points of the Einstein-Hilbert functional is precisely the set of rays in the Sasaki cone with constant scalar curvature. In particular, in this case a Sasakian structure $S=(\xi,\eta,\Phi,g)$ has constant scalar curvature if and only if $(Y,\xi)$ is K-semistable.
\end{theorem}

We can apply this theorem directly to the $S^3_\bfw$-join manifolds $M_{l_1,l_2,\bfw}=M\star_{l_1,l_2}S^3_\bfw$ of \cite{BoToJGA}; however, as shown in Section \ref{sasadmsec} we can obtain a stronger result by direct computation. Recall that the join construction correspond to the product in the Sasaki category. The symmetries of $S^3_\bfw$ are then transfered to the join in some sense and the $S^3_\bfw$--join manifolds $M_{l_1,l_2,\bfw}$ inherits an important two dimensional subcone $\gt^+_\bfw$, called the $\bfw$-cone, of the Sasaki cone $\gt^+$ associated to a two dimensional Lie subalgebra $\gt_\bfw$ of the Lie algebra $\gt$ of a maximal torus lying in the automorphism group of a Sasakian structure. 

In fact, using the relation between $\bfH$ and the Sasaki-Futaki invariant, Lemma \ref{explicitSF} expresses the latter in terms of a rational function obtained explicitly from the $S^3_\bfw$-join construction. This allows us to determine if a ray $\gr_\xi$ in the $\bfw$-cone has constant scalar curvature regardless of whether the transverse scalar curvature vanishes or not. 

\begin{theorem}\label{stabcsc}
Let $M$ be a regular Sasaki manifold with constant transverse scalar curvature, and consider the $S^3_\bfw$-join $M_{l_1,l_2,\bfw}$. Its $\bfw$-cone $\gt^+_\bfw$ has a cscS ray $\gr_\xi$  if and only if the Sasaki-Futaki invariant $\bfF_\xi$ vanishes on the Lie algebra $\gt_\bfw\otimes \bbc$. Then for $\xi\in\gt^+_\bfw$ the polarized affine cone $(Y,\xi)$ associated to $M_{l_1,l_2,\bfw}$ is K-semistable if and only if the Sasakian structure $S=(\xi,\eta,\Phi,g)$ on $M_{l_1,l_2,\bfw}$ has constant scalar curvature (up to isotopy). 
\end{theorem}

In Section \ref{exsec} we give some examples of the manifolds $M_{l_1,l_2,\bfw}$. In particular, Example \ref{bothscal0csc} gives infinitely many contact structures on the two $S^3$-bundles over a Riemann surface of genus greater than one which have both vanishing transverse scalar curvature and Sasaki-Futaki invariant. The topology of the manifolds $M_{l_1,l_2,\bfw}$ has been studied in \cite{BoToJGA,BoToNY}. In particular, if $M$ is simply connected so is $M_{l_1,l_2,\bfw}$ and a method for describing the cohomology ring is given, and is computed in special cases. For example, the integral cohomology ring for Example \ref{tripleCSC} below is computed in \cite{BoToNY} for the $c_1(D)=0$ case ($l_1=1, l_2=w_1+w_2$).\\

Another application of Theorem~\ref{theoCRIT=zero} is to ensure the existence of a Reeb vector field for which the Sasaki--Futaki invariant vanishes identically in some case. 
\begin{theorem}\label{theoMINexist} Assume that the total transversal scalar curvature is sign definite and bounded away from $0$ on each transversal set of the Sasaki cone, then there exists at least one Reeb vector field for which the Sasaki--Futaki invariant vanishes identically.   
\end{theorem}

This hypothesis is fulfilled on contact toric manifolds of Reeb type and when $c_1(D)=0$, but for these cases Theorem~\ref{theoMINexist} was already known \cite{FutakiOnoWang,moi2,reebMETRIC,MSYvolume}. 
 
 Finally, we study the second variation of the Einstein--Hilbert functional. This functional being invariant by rescaling cannot be convex in the usual sense. Let us say that it is {\it transversally convex} (respectively concave) at $\xi$ if for every variation $\xi+ta$ in the Sasaki cone 
 $$\frac{d^2}{dt^2}\bfH(\xi+ta)\geq 0 \;\;\; (\mbox{resp. } \leq 0)$$ 
 with equality if and only if $a\in\R\xi$. We get, as a consequence of a theorem of Matsushima, that $\bfH$ is convex (transversally to rescaling) near rays of Sasaki-$\eta$-Einstein metrics.    
 
 \begin{theorem}\label{theolocCVX}
  Let $(M,D,J,g, \xi)$ be either a cscS compact manifold of negative transverse scalar curvature or a compact Sasaki-$\eta$--Einstein manifold of positive transverse scalar curvature. Then $\bfH$ is transversally convex at $\xi$ (or transversally concave in the negative cscS case with $n$ odd). 
\end{theorem}

\begin{rem} In the Sasaki-Einstein case, our result here can be compared with the convexity results in \cite{MSYvolume}.
\end{rem}

 In Section~\ref{secBG} we recall the basic notions of Sasaki geometry needed later. In Section~\ref{secEH} we prove Lemma~\ref{lemmaCRIT} from which follows Theorem~\ref{theoCRIT=zero} and compute the second variation of $\bfH$. In Section~\ref{secAppli} we apply Lemma~\ref{lemmaCRIT} to prove Theorems~\ref{theoSTABcrit},\ref{critcscS},\ref{theoMINexist},\ref{theolocCVX}. In Section \ref{joinadd} we explicitly compute the Einstein-Hilbert functional and thus the Sasaski-Futaki invariant in terms of rational functions, Lemmas \ref{Hwcone} and \ref{explicitSF}, for the $S^3_\bfw$ join construction of \cite{BoToJGA} which proves Theorem \ref{stabcsc}. We then apply our computations to several examples. 
 
\begin{ack}
This work has benefited from a visit by the third and fourth authors to the University of New Mexico to whom they thank for support and hospitality. The authors also thank Vestislav Apostolov, Tristan Collins, Gabor Sz\'ekelyhidi, and Craig van Coevering for helpful discussions.
\end{ack}

\section{Background}\label{secBG}
\subsection{The Sasaki cone}
We give in this section the basic definitions and facts we need for our purposes. We keep the notation of~\cite{BG:book} and we refer to it for an extensive study of Sasakian geometry.

A Sasakian manifold of dimension $2n+1$ is a Riemannian manifold $(M,g)$ together with a co-oriented dimension $1$ contact distribution $D$ such that the cone metric $\hat{g}$ induced by $g$ on $C(M)$, a connected component of the annihilator, $D^0$, of $D$ in $T^*M$, is K\"ahler with respect to the symplectic structure $\omega$ coming from the inclusion $C(M)\subset T^*M$. Note that we have an inclusion $\iota_g:M\hookrightarrow C(M)$ determined by $S_g(T^*M)\cap C(M) \simeq M$ where $S_g(T^*M)$ is the set of covectors of norm $1$. Usually, one identifies $$\left(C(M)= M\times \R_{>0}, \, \hat{g}=dr\times dr +r^2g,\, \omega=\frac{1}{2}dd^cr^2\right)$$ by defining the map $r: C(M)\ra \R_{>0}$ as $r^2_p=\hat{g}_p(y,y) = g_{\pi(p)}(p,p)$ where $y$ is the vector field induced by dilatation along the fibers in $C(M)\subset T^*M$ and $\pi: T^*M\ra M$.

The integrable complex structure $\hat{J}$ on $Y=C(M)$ determines a CR-structure $(D,J)$ on $M$ as $D= T\iota_g(M)\cap \hat{J}(T\iota_g(M))$ and a Reeb vector field $\xi=\hat{J}y$ which pull back as a Killing vector field on $(M,g)$ transverse to the distribution $D$.    

Let $(M,g,\xi, D)$ be a Sasakian manifold of dimension $2n+1$, where $g$ denotes the Riemannian metric, $\xi$ the Reeb vector field and $D$ the contact structure. The Sasakian structure is also determined by the CR-structure $(D,J)$ together with the contact form $\eta\in \Omega^1(M)$ so that $\eta(\xi)=1$, $\mL_{\xi}\eta =0$, $\ker \eta = D$ and\footnote{In this equation we use the convention in \cite{BGS} rather than \cite{BG:book} since it coincides with the usual convention in K\"ahler geometry.}
$$g= d\eta(\cdot,\Phi(\cdot)) + \eta\otimes \eta$$ 
where $\Phi\in \Gamma(\mbox{End(TM)})$ is defined as $\Phi(\xi)=0$ and $\Phi_{|_{D}}=J$. 

 The Reeb vector field $\xi$ lies in the Lie algebra $\mathfrak{cr}(D,J)$ of $CR$--diffeomorphism $CR(D,J)$. Recall that the space of Sasaki structures sharing the same CR-structure, denoted $\Sas(D,J)$, is in bijection with the cone of Reeb vector fields $\mathfrak{cr}^+(D,J)= \{X \in \mathfrak{cr}(D,J)\,|\, \eta (X)>0\}$. The map $\mathfrak{cr}^+(D,J)\ra \Sas(D,J)$ is given by \begin{equation}\label{defnSasakiCONE}
\xi' \mapsto \left(\frac{\eta}{\eta(\xi')},D,J\right).                                                                                                                                                                                                                                                                                                                                                                                              \end{equation}

From~\cite{BG:book} we know that $\mathfrak{cr}^+(D,J)$ is an open convex cone in $\mathfrak{cr}(D,J)$, invariant under the adjoint action of $\mathfrak{CR}(D,J)$. Moreover, the following result will be useful for our study. 

\begin{theorem}\cite{BGS} Let $M$ be a compact manifold of dimension $2n + 1$ with a
CR-structure $(D, J)$ of Sasaki type. Then the Lie algebra $\mathfrak{cr}(D,J)$ decomposes as
$\mathfrak{cr}(D,J)=\kt_k +\mathfrak{p}$, where $\kt_k$ is the Lie algebra of a maximal torus $T_k$ of dimension $k$
with $1 \leq k \leq n + 1$ and $\mathfrak{p}$ is a completely reducible $\kt_k$--module. Furthermore, every
$X \in \mathfrak{cr}^+(D,J)$ is conjugate to a positive element in the Lie algebra $\kt_k$. 
\end{theorem}

The {\it Sasaki cone} is the set $\kt_k^+=\kt_k \cap  \mathfrak{cr}^+(D,J)$. When seeking extremal or csc Sasaki metrics one deforms the contact structure by $\eta\mapsto \eta +d^c\varphi$ where the function $\varphi$ is invariant under the maximal torus $T_k$. Thus, the Sasaki cone is associated with an isotopy class of contact structures of Sasaki type that is invariant under $T_k$.

\subsection{The Sasaki-Futaki invariant} One can consider the class $\Sas(\xi)$ of Sasakian structures having the same Reeb vector field. 
Let $L_\xi$ be the line bundle having $\xi$ as a section. The inclusion $L_{\xi} \hookrightarrow TM$ induces a sequence of bundle morphisms $$0\ra L_{\xi} \hookrightarrow TM \ra Q_\xi \ra 0.$$ For any CR-structure $(D,J)$ in $\Sas(\xi)$, the restriction $D\ra Q_\xi$ is an isomorphism and provides a complex structure $\bar{J}$ on $Q_\xi$. One can consider the subclass of structures $\Sas(\xi,\bar{J})$ making the following diagram commutes.
   \begin{align}
    \begin{split}
       \xymatrix{\relax
           & TM \ar[d]_{\Phi} \ar[r] &  Q_{\xi}\ar[d]_{\bar{J}}\\
         &TM \ar[r] &  Q_{\xi}
       }
   \end{split}
  \end{align} With that comes a natural notion of transversal holomorphic vector fields $\mathfrak{h}(\xi,\bar{J})$ see~\cite{BGS}.  
  
For a given Sasakian manifold, $(M,g,\xi, D,J)$, the transversal K\"ahler geometry refers to the geometry of $(D,J,g_{|_D})$. More precisely, $M$ is foliated by the Reeb flow. So there are local submersions $\pi_\alpha : U_\alpha \ra V_\alpha$, where $U_\alpha$ and $V_\alpha$ are open subsets of $M$ and $\bC^n$ respectively, such that $\pi_\alpha^* i= \Phi$. In particular, $d\pi_\alpha : (D_{|_{U_\alpha}},J) \ra (TV_{\alpha}, i)$ is an isomorphism and the Sasaki metric is sent to a K\"ahler structure on $V_\alpha$ with a connection $\nabla_\alpha^T$ and curvatures $R_\alpha^T$, $Ric_\alpha^T$, $\rho^T_\alpha$ $s^T_\alpha$...  Since, $\pi_\alpha^*\nabla_\alpha^T$ and $\pi_\beta^*\nabla_\beta^T$ coincide on $U_\alpha\cap U_\beta$, these objects patch together to define global objects on $M$, the {\it transversal} connection and curvatures $\nabla^T$, $R^T$, $Ric^T$, $\rho^T$, $s^T$... See ~\cite{BG:book,FutakiOnoWang} for more details. These tensors are basic, notably the transversal Ricci form $ \rho^T$ 
satisfies $$ \rho^T(\xi,\cdot) =0, \qquad \mL_\xi \rho^T=0$$ and lies in the basic first Chern class $2\pi c^B_1(\calf_\xi)$.    

Since the exterior derivative preserves this condition, the graded algebra of basic forms is a sub-complex of the de Rham complex. Moreover, one can define the basic exterior derivative $d_B$ as the restriction of the differential to these forms, its adjoint $\delta_B$ and the basic Laplacian $\Delta_B=d_B\delta_B+\delta_Bd_B$. The Hodge Theorem holds for the basic cohomology in this context; for a Sasaki metric $g$ there exist a unique basic function $\psi_g$ (of mean value $0$) such that $$\rho^T = \rho^T_H + i\del\overline{\del} \psi_g$$ where $\rho^T_H$ is $\Delta_B$--harmonic. Note that $$\Delta_B\psi_g = s^T - \frac{{\bfS}_{\xi}}{\bfV_\xi}$$ where $\bfV_\xi$ is the volume of $(M,g)$, the volume form is $dv_\xi= \eta\wedge(d\eta)^n$ and $\bfS_{\xi}$ is the total transversal scalar curvature. The volume of the Sasakian manifold $(M,D,J,\xi)$ does not depend on the chosen structure in $\Sas(\xi)$, see~\cite{BG:book} and the total transversal scalar curvature does not depend on the chosen structure in $\Sas(
\xi,\overline{J})$, see~\cite{FutakiOnoWang}.    

Fixing a CR-structure $(D,J)$ on $M$, the Sasaki--Futaki invariant of a Reeb vector field $\xi$ is the map $ \bfF_\xi : \mathfrak{cr}(D,J) \longrightarrow \R$ defined by
%
$$\bfF_\xi(X)= \int_MX.\psi_g dv_g.$$ In fact, $\bfF_\xi$ is more naturally defined on $\mathfrak{h}(\xi,\overline{J})$ and does not depend on the chosen structure in $\Sas(\xi,\overline{J})$, see~\cite{BGS,FutakiOnoWang}, but in this note, to ease the computations we often fix a CR-structure. 

We have $\bfF_\xi([X,Y])=0$, $\bfF_\xi(\xi)=0$ and whenever there is a compatible constant scalar curvature Sasakian metric in $\Sas(\xi,\overline{J})$ then $$\bfF_\xi \equiv 0.$$ 
In order to simplify the notation, we omit explicit reference to $\bar{J}$ in the notation of $\bfF_\xi$ or $\bfS_\xi$ even though they depend on that structure.

 \section{The Einstein--Hilbert functional}\label{secEH}
  \subsection{The first variation}\label{ss1stVAR}
  We define the Einstein--Hilbert functional \begin{equation} \bfH(\xi) = \frac{\bfS^{n+1}_{\xi}}{\bfV_\xi^n}\end{equation} as a functional on the Sasaki cone. Note that $\bfH$ is homogeneous since the rescaling $\xi\mapsto \frac{1}{\lambda}\xi$ gives $$\lambda^{n+1}dv_g\mbox{ and } \frac{1}{\lambda}s^T_g.$$
  
The following lemma clarifies the link between the Einstein--Hilbert functional and the Sasaki-Futaki invariant.

  \begin{lemma}\label{lemmaCRIT} Given $a\in T_\xi\kt_k^+$, we have $$d\bfH_\xi (a) = \frac{n(n+1)\bfS^{n}_{\xi}}{\bfV_\xi^n} \bfF_{\xi}(\Phi(a)).$$ If $\bfS_{\xi} =0$ then $d\bfS_{\xi}= n \bfF_{\xi}(\Phi(a)).$ 
\end{lemma}

 Theorem~\ref{theoCRIT=zero} is a consequence of Lemma~\ref{lemmaCRIT} together with the facts that $\mathfrak{cr}^+(D,J)$ is open in $\mathfrak{cr}(D,J)$, that any element in $\mathfrak{cr}^+(D,J)$ is conjugate to an element of $\kt_k^+$ and that $\bfF_\xi([X,Y])=0$.

\begin{proof}[Proof of Lemma~\ref{lemmaCRIT}]
 Given a path $\xi_t$ in $\kt_k^+$ such that $\xi_0=\xi$, we denote $$\eta_t=\frac{\eta}{\eta(\xi_t)}, \; \Phi_t= \Phi - \Phi\xi_t \otimes \eta_t\; \mbox{ and } $$ 
  \begin{equation}
  \begin{split}
 g_t &= d\eta_t (\Phi_t(\cdot),\cdot) +\eta_t\otimes\eta_t\\
 &= \frac{d\eta(\Phi_t(\cdot),\cdot)}{\eta(\xi_t)} -\frac{d\eta(\xi_t)\wedge  \eta}{\eta(\xi_t)^2}(\Phi_t(\cdot),\cdot)+\eta_t\otimes\eta_t.
 \end{split}
\end{equation} 

 Put $a=\dot{\xi}=(\frac{d}{dt}\xi_t)_{|_{t=0}}$, so $\dot{\eta}= -\eta(a)\eta$, $\dot{\Phi} = -\Phi a \otimes \eta$ and 
  \begin{equation*}
  \begin{split}
   \frac{d}{dt}g_t = & \frac{d\eta(\cdot,\dot{\Phi_t}(\cdot))}{\eta(\xi_t)} - \eta(\dot{\xi_t}) \frac{d\eta(\cdot,\Phi_t(\cdot))}{\eta(\xi_t)^2} - \frac{d\del_t(\eta(\xi_t))\wedge  \eta}{\eta(\xi_t)^2}(\cdot,\Phi_t(\cdot))\\
   &- \frac{d(\eta(\xi_t))\wedge  \eta}{\eta(\xi_t)^2}(\cdot,\dot{\Phi_t}(\cdot)) +2\del_t(\eta(\xi_t))\frac{d(\eta(\xi_t))\wedge  \eta}{\eta(\xi_t)^3}(\cdot,\Phi_t(\cdot))\\
   &\qquad - 2\frac{\del_t(\eta(\xi_t))}{\eta(\xi_t)^3}\eta\otimes \eta.
  \end{split}
\end{equation*}
Now, using the fact that $\eta(\xi_0)=1$ (and, thus, $d(\eta(\xi_0))=0$), we get 
 \begin{equation}\label{variation_met}
  \begin{split}
   \dot{g} &= - d\eta(\cdot,(\Phi a)\otimes \eta(\cdot))  -\eta(a) d\eta(\cdot,\Phi(\cdot)) - (d\eta(a)\wedge \eta) (\cdot,\Phi(\cdot)) - 2\eta(a)\eta \otimes \eta\\
   & = -\eta(a) g - \eta(a)\eta \otimes \eta + b
  \end{split}
\end{equation}
where  \begin{equation}
  \begin{split}
b(X,Y) &:= -\eta(Y)d\eta(X,\Phi a) +\eta(X) (\Phi Y).\eta(a)\\
&= \eta(Y)d\eta(\Phi X,a) + \eta(X)d\eta(\Phi Y, a).   
  \end{split}
  \end{equation}
 
Moreover, $dv_g =\frac{1}{n!}\eta\wedge (d\eta)^n$, thus 
  \begin{equation}\label{varVOLform}
   \left(\frac{d}{dt} dv_{t}\right)_{t=0} = -(n+1)\eta(a) dv_{g_0}.
  \end{equation} The scalar curvature differs from the transversal scalar curvature by a constant  \begin{equation}
   s_g=s_g^T -2n
  \end{equation} following the formulas $Ric(X,\xi)=2n\eta(X)$ and $Ric(Y,Z) =Ric^T(Y,Z) -2g(Y,Z)$ for $X\in \Gamma(TM)$ and $Y,Z \in \Gamma(D)$ see~\cite[Theorem 7.3.12]{BG:book}. As computed in~\cite[p.63]{besse} we have $$\int_M\dot{s_g} dv_g = \int_M ( \Delta(\mbox{tr}_g \dot{g} + \delta(\delta\dot{g}))- g(\dot{g},Ric_g))dv_g = -\int_Mg(\dot{g},Ric_g)dv_g$$ where $\delta$ denotes the co-adjoint of the Levi-Civita connection on the space of symmetric tensors. Since $\int_M\delta(\beta) dv_g =0$ and $\int_M \Delta(\beta) dv_g=0$ for any tensor $\beta$, combining~\eqref{variation_met} with the last formula we have 
  \begin{equation}\label{INTvars}
  \begin{split}
   \int_M\dot{s_g} dv_g &= \int_M (\eta(a)(s_g + g(\eta\otimes \eta,Ric_g)) - g(b,Ric_g))dv_g\\
   &= \int_M (\eta(a)(s_g +2n) - g(b,Ric_g))dv_g\\
   &= \int_M \eta(a)s^T_g dv_g
   \end{split}
  \end{equation} for the second line we used the identity $Ric_g(\xi,\xi) =2n$ recalled above and, for the last line, the fact that $g(b,Ric_g)=0$ that we shall now prove. Observe that $b$ is symmetric and that at each point $p\in M$, $b(\xi,\xi)=0$, $b(u,v)=0$ if $u,v\in D_p$, while $Ric(u,\xi) = 0$ at $p$ whenever $u\in D_p$. Writing $b$ and $Ric_g$ with respect to an orthonormal basis $\xi, e_1,Je_1,\dots, e_n,Je_n$ of $T_pM$ with $e_i \in D_p$ we see that \begin{equation}\label{b_perp_ric} g_p(b,Ric_g)=0.
  \end{equation}

  \begin{rem}\label{remETAabasic}
   We have $\mL_a\eta = 0$. Indeed, $\mL_{\xi_t}\eta_t = 0$ and thus $$0= \mL_a\eta+ \mL_\xi(-\eta(a)\eta) = \mL_a\eta -\mL_\xi(\eta(a))\eta = \mL_a\eta$$ since $a\in T_\xi\kt_k$ (i.e $[a,\xi]=0$), so $$\mL_\xi(\eta(a)) = \xi.(\eta(a)) = d\eta(\xi,a)+ a.\eta(\xi) +\eta([a,\xi]) =d\eta (\xi,a)=0.$$
  \end{rem}
Putting the variational formulas \eqref{varVOLform} and \eqref{INTvars} together we get  
   \begin{equation}\label{INTvarST}
  \begin{split}
   \frac{d}{dt} \left(\bfS_{\xi_t}\right)_{t=0}  &= \int_M (\dot{s_g}dv_g  + s_g^T\dot{dv_g})\\
   &= \int_M (\eta(a) s_g^T dv_g -(n+1)\eta(a)s_g^Tdv_g)\\
   &= -n\int_M\eta(a)s_g^Tdv_g.
  \end{split}
\end{equation} Hence, using \eqref{INTvarST} and \eqref{INTvars} we have
  \begin{equation}\label{INTvarH}
  \begin{split}
   \left(\frac{d}{dt}\bfH(\xi_t) \right)_{t=0} &= \frac{\bfS^{n}_{\xi}}{\bfV_\xi^{n+1}} \left( (n+1)\bfV_\xi \dot{\bfS_{\xi}} -n\bfS_{\xi} \dot{\bfV_\xi}\right)\\
   & = -n(n+1) \frac{\bfS^{n}_{\xi}}{\bfV_\xi^n} \left( \int_M \eta(a)\left(s_g^T -\frac{\bfS_{\xi}}{\bfV_\xi}\right)dv_g\right)\\
   & = -n(n+1) \frac{\bfS^{n}_{\xi}}{\bfV_\xi^n} \left( \int_M \eta(a)\Delta_B\psi_gdv_g\right)
  \end{split}
\end{equation} where $\Delta_B$ is the basic Laplacian and $\psi_g$ is the unique (normalized to have mean value $0$) basic function satisfying $\rho^T= \rho^T_H +i\del\bar{\del} \psi_g.$ Note that $\eta(a)$ is basic by Remark~\ref{remETAabasic}. We continue  

  \begin{equation}\label{INTvarlast}
  \begin{split}
   \int_M \eta(a)\Delta_B\psi_gdv_g& = \int_M g(d_B\eta(a), d_B\psi_g)dv_g \\
   &= \int_M g(d\eta(a), d\psi_g)dv_g\\
   &=\int_M (\nabla^g\psi_g).\eta(a)dv_g\\
    &=\int_M( (\mL_a\eta)(\nabla^g\psi_g) - d\eta(a, \nabla^g\psi_g)) dv_g\\
    &= -\int_M d\eta(a, \nabla^g\psi_g) dv_g\\
    &= -\int_M (g(\Phi(a), \nabla^g\psi_g) - \eta(\Phi(a))\eta(\nabla^g\psi_g)) dv_g\\
    &= -\int_M g(\Phi(a), \nabla^g\psi_g)dv_g\\
    &= -\int_M \Phi(a).\psi_g \,dv_g\\
    &= -\bfF_\xi(\Phi(a)).
  \end{split}
\end{equation}
Combining \eqref{INTvarlast}, \eqref{INTvarH} and \eqref{INTvarST}, we get Lemma~\ref{lemmaCRIT}.\end{proof}

\subsection{The second variation}\label{ss2ndVAR}


To simplify the expression of the second variation of the Einstein--Hilbert functional, we use the standard notation for inner product on the space $L^2(M)$ of square integrable real valued functions on $M$ $$\langle f,h \rangle := \int_M f h dv_g \;\;\mbox{ and }\;\; \|f\|^2:= \int_M f^2dv_g$$ and for vector field $X\in\Gamma(TM)$ or one form $\beta\in\Gamma(T^*M)$ $$\|X\|^2:= \int_M g(X,X)dv_g \;\;\mbox{ and }\;\; \|\beta\|^2:= \int_M g(\beta,\beta)dv_g.$$ Moreover we define the normalized transversal scalar curvature $$\mathring{s}^T =s^T - \frac{\bfS}{\bfV}$$ which integrates to zero.

\begin{lemma}\label{2ndvar} 
For each $a\in T_\xi \kt_k^+$ and variation $\xi_t =\xi +ta$, we have   
  \begin{equation*}\begin{split}
  \frac{d^2}{dt^2}\bfH(\xi_t)_{t=0} &= n(n+1)(2n+1)\frac{\bfS^n_\xi}{\bfV^n_\xi}\|d(\eta(a))\|^2 - n(n+1)\frac{\bfS^{n+1}_\xi}{\bfV^{n+1}_\xi} \|\eta(a)\|^2 \\
  & \;\;\;\;+n(n+1)^2\frac{\bfS^n_\xi}{\bfV^n_\xi}\int_M \eta(a)^2 \mathring{s}^Tdv_g\\
  & \;\;\;\; +  n(n+1)\frac{\bfS^{n-1}_\xi}{\bfV^{n+2}_\xi} \left(n\bfV\langle \mathring{s}^T, \eta(a) \rangle - \bfS\langle \eta(a),1 \rangle \right)^2.   \end{split}
 \end{equation*}
\end{lemma}

\begin{proof}
Considering that expression $$\frac{d}{dt} \bfH(\xi_t)_{t=0} = -n(n+1)\frac{\bfS^n_\xi}{\bfV^n_\xi}\int_M \eta(a) \mathring{s}^Tdv_g$$ we have that 
 
 \begin{equation}\label{secondVAR}\begin{split}
 \frac{-1}{n(n+1)}\frac{d^2}{dt^2}\left(\bfH(\xi_t)\right)_{t=0} &= n\left( \frac{\bfS^{n-1}_\xi\dot{\bfS}}{\bfV^n_\xi} -  \frac{\bfS^{n}_\xi\dot{\bfV}}{\bfV^{n+1}_\xi} \right) \int_M \eta(a) \mathring{s}^Tdv_g \\
 &\;\;\;-\frac{\bfS^n_\xi}{\bfV^n_\xi}\int_M \eta(a)^2  \mathring{s}^Tdv_g + \frac{\bfS^n_\xi}{\bfV^n_\xi}\int_M \eta(a)  \dot{s}^T dv_g \\
 &\;\;\;-\frac{\bfS^n_\xi}{\bfV^n_\xi}\int_M \eta(a)^2  \left(\frac{\dot{\bfS}}{\bfV_\xi} -  \frac{\bfS_\xi\dot{\bfV}}{\bfV^{2}_\xi}\right)dv_g \\
 &\;\;\;- (n+1)\frac{\bfS^n_\xi}{\bfV^n_\xi}\int_M \eta(a)^2 \mathring{s}^Tdv_g.
 \end{split}\end{equation}

The formulas~\eqref{varVOLform} and \eqref{INTvarST} can be used again here to see that  
 \begin{equation*}\begin{split}\bfV \dot{\bfS} - \bfS\dot{\bfV} &= -n \bfV\int_M \eta(a)s^Tdv_g + (n+1)\bfS\int_M \eta(a)dv_g\\ 
 &= -n\bfV \int_M \eta(a)\mathring{s}^Tdv_g + \bfS\int_M \eta(a)dv_g\end{split}
 \end{equation*}  so that, summing the first and fourth terms of the right hand side of \eqref{secondVAR} gives 
 \begin{equation}\label{1and4termsSECvar}\begin{split} 
 -n^2&\frac{\bfS^{n-1}_\xi}{\bfV^{n}_\xi}\left(\int_M\eta(a) \mathring{s}^Tdv_g\right)^2 + n \frac{\bfS^{n}_\xi}{\bfV^{n+1}_\xi}\int_M\eta(a) dv_g\int_M\eta(a) \mathring{s}^Tdv_g \\
 &+  n \frac{\bfS^{n}_\xi}{\bfV^{n+1}_\xi}\int_M\eta(a) dv_g\int_M\eta(a) \mathring{s}^Tdv_g -  \frac{\bfS^{n+1}_\xi}{\bfV^{n+2}_\xi}\left(\int_M\eta(a) dv_g\right)^2 \\
 &= - \frac{\bfS^{n-1}_\xi}{\bfV^{n+2}_\xi} \left(n\bfV\langle \mathring{s}^T, \eta(a) \rangle - \bfS\langle \eta(a),1 \rangle \right)^2.
 \end{split}
 \end{equation}

 For the third term of \eqref{secondVAR}, that is $\frac{\bfS^n_\xi}{\bfV^n_\xi}\int_M \eta(a)  \dot{s}^T dv_g$, we need to work a little bit more. First, as before $s_g=s_g^T -2n$ so that, again by~\cite[p.63]{besse}, 
\begin{equation}\label{variation_curv}\dot{s^T_g} = \dot{s_g} =\Delta(\mbox{tr}_g \dot{g}) + \delta(\delta\dot{g})- g(\dot{g},Ric_g)\end{equation} and, see~\eqref{variation_met}, $\dot{g}=-\eta(a) g - \eta(a)\eta \otimes \eta + b$. 

For the first term of~\eqref{variation_curv}, we compute that $\mbox{tr}_g \dot{g} = -2(n+1)\eta(a)$ and we get 
\begin{equation}\label{1stTERM}\begin{split}
\int_M\eta(a) \Delta(\mbox{tr}_g \dot{g} )dv_g &= -2(n+1) \int_M\eta(a)\Delta\eta(a)dv_g \\
&= -2(n+1) \int_M|d\eta(a)|^2 dv_g. \end{split}\end{equation}

For the second term of~\eqref{variation_curv} we need to understand a little bit better the operator $\delta: S^p(T^*M) \ra S^{p-1}(T^*M)$. It is defined as the co-adjoint of the Levi-Civita connection seen as an operator on symmetric tensors. One can check the following two rules for every symmetric $p$--tensor $\beta\in S^p(T^*M)$, $1$--form $\alpha\in\Gamma(T^*M)$ and function $f \in C^1(M)$, \begin{equation}\label{rules_delta}\begin{split} &\delta (f\beta) = f\delta \beta - \beta(\nabla f, \cdot, \dots,\cdot) \\ 
&\delta(\alpha\otimes \alpha) = 2\delta(\alpha) \alpha. \end{split}\end{equation} Moreover, on $1$--forms $\delta$ coincides with the codifferential, so that $$\int_M\eta(a) \delta(\delta(\dot{g}))dv_g = \int_M g(d\eta(a),\delta(\dot{g}))dv_g.$$ 
Now, by checking that for each $u,v\in D$, $b(\xi,u)= - g(a,u)$, $b(u,v)=0$ and $b(\xi,\xi)=0$, we can rewrite \begin{equation}
b = - g(a,\cdot)\otimes \eta - \eta\otimes g(a,\cdot) -2\eta(a) \eta^2
\end{equation} 
Hence, using~\eqref{rules_delta} we have
\begin{equation}\label{delta_b}\delta b  = - 2(\delta(\eta)g(a,\cdot) + \delta(g(a,\cdot))\eta)  -4\eta(a)\delta(\eta)\eta +2\eta(\nabla(\eta(a)))\eta.
\end{equation} The following (pointwise) identity is consequence of Remark~\ref{remETAabasic} 
\begin{equation} \label{identityETA}
g(d\eta(a),\eta) = \mL_\xi \eta(a)= 0.
\end{equation} This last rule simplifies our problem a lot. Indeed, applying it with formula \eqref{delta_b}, we have 
$$g(d\eta(a),\delta b) = - \delta(\eta) g(d\eta(a), g(a,\cdot)) =- \delta(\eta)\mL_a\eta(a)$$ and then  
\begin{equation}\label{detaa_delta_b=0} \begin{split}\int_Mg(d\eta(a),\delta b) dv_g &= -\int_M \delta(\eta)\mL_a\eta(a) dv_g \\
&= -\int_M g(\eta, d\mL_a\eta(a)) dv_g \\
&= -\int_M \mL_\xi\mL_a\eta(a) dv_g =0\end{split}\end{equation}
since $\mL_\xi\mL_a\eta(a) = \mL_a\mL_\xi\eta(a) -\mL_{[\xi,a]} \eta(a)=0$ because $\mL_\xi\eta(a)=0$ and $[\xi,a]=0$. The identity \eqref{identityETA} implies also that
\begin{equation}
 g(d\eta(a),\delta (\eta\otimes \eta)) = 2\delta(\eta) g(d\eta(a),\eta) =0
\end{equation} which, combined with ~\eqref{detaa_delta_b=0} and that\begin{equation}\begin{split}
\int_M \eta(a) g(d\eta(a),\delta g)dv_g &=\frac{1}{2}\int_M  g(d\eta(a)^2,\delta g)dv_g \\
&= \frac{1}{2}\int_M g(D d(\eta(a)^2), g)dv_g\\
&= \frac{-1}{2} \int_M \Delta \eta(a)^2 dv_g =0
\end{split}\end{equation} gives the second term of~\eqref{variation_curv} as 
\begin{equation}\label{2sdTERM_varCURV}\begin{split}
 \int_M\eta(a)\delta(\delta\dot{g})dv_g &= \int_M g(d\eta(a),\delta\dot{g})dv_g \\
 &= \int_M g(d\eta(a),\delta(-\eta(a) g - \eta(a)\eta \otimes \eta + b))dv_g\\
 &= \int_M (g(d\eta(a), g(\nabla \eta(a), \cdot)) +\eta(\nabla(\eta(a))) g(d\eta(a),\eta))dv_g\\
 &= \int_M g(d\eta(a), g(\nabla \eta(a), \cdot))dv_g\\
 &= \int_M |d\eta(a)|_g^2dv_g
 \end{split}
\end{equation}

Following the same lines than \eqref{INTvars}, the last term of~\eqref{variation_curv}, when integrated with $\eta(a)$ gives 
\begin{equation}\label{3TERM_varCURV}
\int_M\eta(a)^2s^T_gdv_g.
\end{equation}
 In summary we have 
 \begin{equation} \label{thirdTERMsecondVAR}
 \int_M \eta(a)  \dot{s}^T dv_g= \int_M \left(-(2n+1)|d(\eta(a))|^2  +\eta(a)^2s_ g^T\right)dv_g.
 \end{equation}
 
 Putting \eqref{1and4termsSECvar} and~\eqref{thirdTERMsecondVAR} in~\eqref{secondVAR} we get  
  
 \begin{equation}\label{secondVAR2}\begin{split}
 \frac{-1}{n(n+1)}\frac{d^2}{dt^2}&\left(\bfH(\xi_t)\right)_{t=0} =-\frac{\bfS^n_\xi}{\bfV^n_\xi}\int_M \eta(a)^2  \mathring{s}^Tdv_g + \frac{\bfS^n_\xi}{\bfV^n_\xi} \int_M \eta(a)^2s_ g^T dv_g  \\
 &\;\;\;- (n+1)\frac{\bfS^n_\xi}{\bfV^n_\xi}\int_M \eta(a)^2 \mathring{s}^Tdv_g  -(2n+1) \frac{\bfS^n_\xi}{\bfV^n_\xi} \|d\eta(a)\|^2\\
 &\;\;\;- \frac{\bfS^{n-1}_\xi}{\bfV^{n+2}_\xi} \left(n\bfV\langle \mathring{s}^T, \eta(a) \rangle - \bfS\langle \eta(a),1 \rangle \right)^2\\
 &= \frac{\bfS^{n+1}_\xi}{\bfV^{n+1}_\xi}\int_M \eta(a)^2 dv_g - (n+1)\frac{\bfS^n_\xi}{\bfV^n_\xi}\int_M \eta(a)^2 \mathring{s}^Tdv_g\\
 &\;\;\; -(2n+1) \frac{\bfS^n_\xi}{\bfV^n_\xi} \|d\eta(a)\|^2\\
 &\;\;\;- \frac{\bfS^{n-1}_\xi}{\bfV^{n+2}_\xi} \left(n\bfV\langle \mathring{s}^T, \eta(a) \rangle - \bfS\langle \eta(a),1 \rangle \right)^2.
 \end{split}\end{equation}
 \end{proof}

In the case of a cscS metric Lemma \ref{2ndvar} gives

\begin{corollary}\label{coroCSCSlocal2ndVAR} Suppose that $(M, g, D, J,\xi)$ is a cscS structure with constant transverse scalar curvature $s^T_g$, Reeb vector field $\xi$, and contact form $\eta$. For each $a\in T_\xi \kt_k^+$ and variation $\xi_t =\xi +ta$, we have   
  \begin{equation*}\begin{split}
  \frac{d^2}{dt^2}_{t=0} \bfH(\xi_t) &= n(n+1)(2n+1)(s^T_g)^n\left( \|d(\eta(a))\|_g^2 -\frac{s^T_g}{2n+1} \|\eta(a)\|_g^2 \right)
   \\
  & \;\; +n(n+1) \frac{(s^T_g)^{n+1}}{\bfV_\xi}\left(\int_M\eta(a)dv_g\right)^2 
  \end{split}
 \end{equation*}
\end{corollary}

\section{Critical Points, Stability, and Local Convexity} \label{secAppli}
\subsection{Sasakian K--stability}
In~\cite{CollinsSz}, Collins and Sz\'ekelyhidi introduced the notion of K-semistability for a general Sasakian manifold in terms of the K-semistability of the K\"ahler cone $(Y=C(M),\hat{J})$, as an affine variety, with respect to the polarization given by the Reeb vector field $\xi$. It turns out that this notion extends the orbifold K-semistability of Ross and Thomas \cite{RoThJDG} to irregular Sasakian structures.

In this note, we use only the simplest aspects of this notion and the proof we give of Theorem~\ref{theoSTABcrit} follows the ideas of the proof of Theorem 6 in~\cite{CollinsSz}. Consequently, we do not discuss the general notion of Sasakian K-semistability and encourage the interested reader to consult~\cite{CollinsSz} for more details. 

The starting ingredients are an affine variety $Y$ with an action of a torus $T_{\C}$ whose maximal compact subtorus $T_\R$ Lie algebra $\kt$ contains the Reeb vector field $\xi$. A $T_{\C}$--equivariant test configuration for $Y$ is given by a a set of $k$ $T_{\C}$--homogeneous generators $f_1,\dots, f_k$ of the coordinate ring of $Y$ and $k$ integers $w_1, \ldots, w_k$. The functions $f_1,\dots, f_k$ are used to embed $Y$ in $\C^k$ on which the integers $w_1, \ldots, w_k$ determine an $\C^*$ action (via the weights). By taking the flat limit of the orbits of $Y$ to $0\in \C$ we get a family of affine schemes $$\mathcal{Y} \longrightarrow \C.$$ There is then an action of $\C^*$ on central fiber $Y_0$, generated by $a \in  \mbox{Lie} (T_{\R}')$, where $T'_\C \subset\mbox{Gl}(k,\C)$ is some torus containing $T_{\C}$. The Donaldson--Futaki invariant of such a test configuration is essentially defined in~\cite{CollinsSz} to be
\begin{equation}\label{DonFutINVdefn}
 Fut(Y_0,\xi,a) := \frac{\bfV_\xi}{n}  D_{a}\left( \frac{\bfS_\xi}{\bfV_\xi} \right)+ \frac{\bfS_\xi D_{a}\bfV_\xi}{n(n+1)\bfV_\xi}.
\end{equation} They define that $(Y,\xi)$ is {\it $K$--semistable} if for each $T_{\C}$, such that $\xi \in \mbox{Lie} (T_{\R})$ and any $T_\C$ equivariant test configuration \begin{equation}\label{StabCOND}Fut(Y_0,\xi,a)\geq 0.
\end{equation} The case of a product configuration is when $\mathcal{Y}= Y\times \C$ and the action of $\C^*$ on $Y_0$ is induced by a subgroup of $T_{\C}$. 

\begin{proof}[Proof of Theorems~\ref{theoSTABcrit} and \ref{critcscS}]
To prove Theorem \ref{theoSTABcrit} we first note that the linearity in $a$ of the right hand side of~\eqref{DonFutINVdefn} implies that if $(Y,\xi)$ is $K$--semistable then $Fut(Y_0,\xi,a)=0$ for every product configuration. A straigthforward computation (done in \cite[Lemma 2.15]{TivC}) shows that there exists a constant $c_n>0$ depending only on the dimension $n$ such that $Fut(Y_0,\xi,a)=c_n \bfF_{\xi}(\Phi(a))$.
So Lemma \ref{lemmaCRIT} gives
$$d_\xi\bfH (a) = \frac{n(n+1)\bfS_\xi^{n}}{c_n\bfV_\xi^{n}}Fut(Y_0,\xi,a)$$ 
from which Theorem~\ref{theoSTABcrit} follows. 

Theorem \ref{critcscS} follows directly from Theorem \ref{theoSTABcrit}, Corollary 1 of \cite{CollinsSz}, and Proposition 5.2 of \cite{BGS}. 
\end{proof}

\subsection{Existence of critical points} \label{secEXISTcrit}

We say that a set $\Sigma$ is a transversal subset of $\kt^+_k$ if $\Sigma\subset\kt^+_k$ and that $\Sigma$ meets each ray passing through $\kt^+_k$ in a single point. In particular, a transversal subset of $\kt^+_k$ is a codimension one relatively compact subset of $\kt_k$ whose closure does not contain $0$. For example, taking any codimension $1$ vector subspace $H\subset \kt_k$ that does not contain $\xi$
the set $$\Sigma_{\xi,H}= \kt_k^+\cap \left\{\left. \xi + a\,\right|\,a \in H\right\}$$ is a transversal subset of $\kt^+_k$.  

\begin{proof}[Proof of Theorem~\ref{theoMINexist}]
Assume that the total transversal scalar curvature is sign definite and bounded in the sense that there exists an order $1$ homogeneous function $m: \kt_k^+ \ra \R_{>0}$, bounded below by a positive number on any a transversal subset of $\kt^+_k$, and such that 
\begin{equation}\label{trscbound}
|\bfS_{\xi}| \geq m_\xi \bfV_{\xi} \qquad \qquad \forall \xi \in \kt_k^+.
\end{equation} 
Another way to state this condition is that the set $\{\xi \in \kt_k^+ \,| \, \bfV_\xi = \bfS_\xi \}$ is relatively compact in $\kt_k$. This condition is fulfilled in the toric case~\cite{moi2}.  

Assuming the bound \eqref{trscbound} and first that $\bfS_{\xi}>0$, we have $$\bfH(\xi) \geq m_{\xi}^{n+1}\bfV_\xi.$$ The right hand side is still a homogeneous function. Consider $\Sigma$, a transversal subset of $\kt^+_k$, given by the projection of a codimension $1$ subspace $H$ transverse to $\xi$ in $T_\xi\kt_k^+$ (i.e $\Sigma=\Sigma_{\xi, H}$ as above). The condition above implies that there exists $m_o>0$ such that $$\bfH(\xi) \geq m^{n+1}_{\xi} \bfV_\xi \geq m_o^{m+1} \bfV_\xi\qquad \qquad \forall \xi\in \Sigma.$$ Now, $\bfV_\xi$ is a strictly convex function on $\Sigma$ that tends to infinity on the boundary $\del \Sigma$ as it is proved in~\cite{MSYvolume}. Thus, as a function on $\Sigma$, $\bfH$ reaches its minimum somewhere in (the relative interior of) $\Sigma$ that point, say $\xi_o$, is a minimum, thus a critical point, of $\bfH$ but $\bfS_{\xi_o}\neq 0$ by hypothesis. Therefore $\bfF_{\xi_o} \equiv 0$ thanks to Lemma~\ref{lemmaCRIT}. 
The case $\bfS_{\xi}< 0$ is similar and theorem~\ref{theoMINexist} follows. 
\end{proof}

\subsection{Local convexity (concavity)}

It is well known \cite{MSYvolume} that the volume functional is convex on the Sasaki cone. Since $\bfH$ is homogenous, it cannot be convex in the usual sense but one can wonder if it is transversally convex. For convenience we shall refer to transverse convexity or transverse concavity by simply convexity or concavity. We also know from \cite{moi2,BoToJGA} that global convexity generally fails in the cscS case. Hence, it is interesting to investigate when local convexity and/or concavity can hold. 

\begin{lemma}\label{conlemm} 
Let $(M^{2n+1},D,J,g, \xi)$ be a cscS compact manifold of non-zero transversal scalar curvature, and $T\subset CR(D,J)\cap \mbox{Isom}(g)$ be the maximal compact torus whose Lie algebra is identified with $\kt_k$.  
\begin{enumerate}
\item If the transverse scalar curvature is negative, the E-H functional $\bfH$ is convex if $n$ is even and concave if $n$ is odd.
\item If the transverse scalar curvature is positive and the first non-zero eigenvalue of the Laplacian, restricted to the space of $T$--invariant functions is bounded below by $\frac{s^T_g}{2n+1}$, then $\bfH$ is convex near $\xi$.     
\end{enumerate}
Moreover, in both cases the cscS metric is isolated in the Sasaki cone.
\end{lemma}

\begin{proof}
Consider the formula for the second variation of $\bfH$ in the Corollary \ref{coroCSCSlocal2ndVAR}. Since $\eta(a)$ is a Hamiltonian we can normalize it so that $\int_M \eta(a)dv_g =0$. Thus, the formula reduces to
\begin{equation}\label{cscLOCALcvxCONDITION2}
 \frac{d^2}{dt^2}_{t=0} \bfH(\xi_t) = n(n+1)(2n+1)(s^T_g)^n\left(\|d(\eta(a))\|^2_g -\frac{s^T_g}{(2n+1)} \|\eta(a)\|^2_g\right). 
\end{equation}
There are two cases to consider.
\newline Case 1. $s^T_g<0$: Then Equation \eqref{cscLOCALcvxCONDITION2} implies that if $n$ is even, $\bfH$ is convex near $\xi$, whereas, if $n$ is odd, $\bfH$ is concave near $\xi$.
\newline Case 2. $s^T_g>0$. In this case using the well known inequality 
$$\|d(\eta(a))\|^2_g\geq \grl_1\|\eta(a)\|^2_g$$ 
we see that $\bfH$ is convex near $\xi$ as soon as $\grl_1>  \frac{s^T_g}{(2n+1)}$.
\end{proof}

\begin{proof}[Proof of Theorem \ref{theolocCVX}]
For negative transverse scalar curvature, the theorem follows immediately from (1) of Lemma \ref{conlemm}. So we consider the positive Sasaki-$\eta$-Einstein case. We give two proofs of this. First we note that a well known result of Tanno says that for every positive Sasaki-$\eta$-Einstein metric there is a transverse homothety to a Sasaki-Einstein metric with $s^T_g=4n(n+1)$ where the Lichnerowicz bound $\grl_1\geq 2n+1$ holds (cf. \cite[Corollary 11.3.11]{BG:book}). Thus, we have $\grl_1\geq 2n+1>\frac{4n(n+1)}{2n+1}=\frac{s^T_g}{2n+1}$. Thus Lemma \ref{conlemm} implies that $\bfH$ is convex near a Sasaki-Einstein metric. But we know that $\bfH$ is invariant under a transverse homothety. So it is convex for any positive Sasaki-$\eta$-Einstein metric. 

For the second proof we use the bound $\grl_1\geq\frac{s^T_g}{2n}$ due to Matsushima Theorem~\cite{matsushima} which holds for any positive transverse K\"ahler-Einstein metric as described in Theorem 3.6.2 in \cite{Gaubook}. The proof given there is local differential geometric in nature and clearly applies to the quasi-regular and irregular Sasakian structures. So the result follows from (2) of Lemma \ref{conlemm}.
\end{proof}  

Example~\ref{tripleCSC} exhibits a (positive) cscS metric for which $\bfH''(\xi)<0$. In particular, the bound $\grl_1>\frac{s^T_g}{2n+1}$ does not hold for this variation. 

Example \ref{criticalptsnoncsc} shows that local convexity (concavity) fails for $\bfH$ when $\bfS =0$ as one can expect. Indeed, the case $\bfS =0$ is a rather special type of critical point of $\bfH$ as it appears in Theorem~\ref{theoCRIT=zero}. It seems reasonable to put that case aside and investigate local convexity/concavity near rays of positive or negative cscS metrics. More generally, it is interesting to contemplate whether the EH--functional is a Morse function away from its zero set.

Example \ref{bothscal0csc} shows that there are cscS metrics with transverse scalar curvature identically zero. In this case the local convexity holds even though $\bfH$ vanishes up to order six at the given point.

\section{The Einstein-Hilbert Functional on the $\bfw$-cone of a Sasaki Join}\label{joinadd}
In this section we apply our results to the $(l_1,l_2)$-join $M_{l_1,l_2,\bfw}=M\star_{l_1,l_2}S^3_\bfw$ of a regular Sasaki manifold $M$ of constant scalar curvature with the weighted 3-sphere $S^3_\bfw$. These manifolds have recently been the object of study by the first and last authors \cite{BoToIMRN,BoToJGA,BoToNY,BoToPaul}. They include an infinite number of homotopy types as well as an infinite number of contact structures of Sasaki type occurring on the same manifold \cite{BoToIMRN,BoToinfcon}.

\subsection{A Sasaki Join}
We present only a brief review and refer to Section 3 of \cite{BoToJGA} for all the details.
Let $M$ be a regular Sasaki manifold which is an $S^1$-bundle over a compact CSC K\"ahler manifold $N$ with a primitive K\"ahler class $[\gro_N]\in H^2(N,\bbz)$.
Let $S^3_\bfw$ be the weighted 3-sphere, that is, $S^3$ with its standard contact structure, but with a weighted contact 1-form whose Reeb vector field generates rotations with generally different weights $w_1,w_2$ for the two complex coordinates $z_1,z_2$ of $S^3\subset \bbc^2$.

Consider the join $M_{l_1,l_2,\bfw}=M\star_{l_1,l_2}S^3_\bfw$, where both $\bfw=(w_1,w_2)$ and $\bfl=(l_1,l_2)$ are pairs of relatively prime positive integers. We can assume that the weights $(w_1,w_2)$ are ordered, namely they satisfy $w_1\geq w_2$. Furthermore, $M_{l_1,l_2,\bfw}$ is a smooth manifold if and only if $\gcd(l_2,l_1w_1w_2)=1$ which is equivalent to $\gcd(l_2,w_i)=1$ for $i=1,2$. Henceforth, we shall assume these conditions. 

The base orbifold $N\times \bbc\bbp^1[\bfw]$ has a natural K\"ahler structure, namely the product structure, and this induces a Sasakian structure $\mathcal{S}_{l_1,l_2,\bfw}=(\xi_{l_1,l_2,\bfw},\eta_{l_1,l_2,\bfw},\Phi,g)$ on $M_{l_1,l_2,\bfw}$. The transverse complex structure $J=\Phi |_{\cald_{l_1,l_2,\bfw}}$ is the lift of  the product complex structure on $N\times \bbc\bbp^1[\bfw]$. The K\"ahler form on $N\times \bbc\bbp^1[\bfw]$ is $\gro_{l_1,l_2}=l_1\gro_N+ l_2\gro_\bfw$ where $\gro_\bfw$ is the standard K\"ahler form on $\bbc\bbp^1[\bfw]$ which satisfies $[\gro_\bfw]=\frac{[\gro_0]}{w_1w_2}$ where $\gro_0$ is the standard volume form on $\bbc\bbp^1$.

We remind the reader that by the $\bfw$-Sasaki cone we mean the two dimensional subcone of Sasaki cone induced by the Sasaki cone of $S^3_\bfw$. It is denoted by $\gt_\bfw^+$ and can be identified with the open first quadrant in $\bbr^2$.

\begin{theorem}\cite{BoToJGA}\label{join data}
Let $M_{l_1,l_2,\bfw}=M\star_{l_1,l_2}S^3_\bfw$ be the join as described above with the induced contact structure $\cald_{l_1,l_2,\bfw}$. Let $\bfv=(v_1,v_2)$ be a weight vector with relatively prime integer components and let $\xi_\bfv$ be the corresponding Reeb vector field in the Sasaki cone $\gt^+_\bfw$ and let
$s=\gcd(|w_2v_1-w_1v_2|,l_2)$. Then the quotient of $M_{l_1,l_2,\bfw}$ by the flow of the Reeb vector field $\xi_\bfv$ is a projective algebraic orbifold written as a the log pair $(S_n,\grD)$ where $S_n$ is the total space of the projective bundle $\bbp(\BOne\oplus L_n)$ over the K\"ahler manifold $N$ with $n=l_1\bigl(\frac{w_1v_2-w_2v_1}{s}\bigr)$, $\grD$ the branch divisor
\begin{equation}\label{branchdiv2}
\grD=(1-\frac{1}{m_1})D_1+ (1-\frac{1}{m_2})D_2,
\end{equation}
with ramification indices $m_i=v_i\frac{l_2}{s}=v_im$ and divisors $D_1$ and $D_2$ given by the zero section $\BOne\oplus 0$ and infinity section $0\oplus L_n$, respectively. The fiber of the orbifold $(S_n,\grD)$ is the orbifold $\bbc\bbp[v_1,v_2]/\bbz_m$.

Moreover, the induced K\"ahler structure $(g_B,\gro_B)$ on $(S_n,\grD)$ may be chosen to be admissible, with K\"ahler class given by $r= \frac{w_1v_2-w_2v_1}{w_1v_2+w_2v_1}$, and satisfies
$$g^T =\frac{l_2}{4\pi}\frac{\pi_\bfv^* g}{mv_1v_2}=\frac{\pi_\bfv^* g_B}{mv_1v_2},\qquad d\eta_\bfv=\frac{l_2}{4\pi} \frac{\pi_\bfv^*\gro}{mv_1v_2} =\frac{\pi_\bfv^*\gro_B}{mv_1v_2}=\gro^T$$
where $g^T=d\eta\circ (\BOne\otimes \Phi)$.

\end{theorem}

\subsection{The Einstein-Hilbert Functional}


We compute the Einstein-Hilbert functional on the $\bfw$-cone of the join $M_{l_1,l_2,\bfw}=M\star_{l_1,l_2}S^3_\bfw$. An arbitrary element of $\gt^+_\bfw$ takes the form $\xi_\bfv=v_1H_1+v_2H_2$ where $H_i$ are the restrictions to $M_{l_1,l_2,\bfw}$ of the infinitesimal generators of the rotation $z_i\mapsto e^{i\theta}z_i$ on $S^3$, and let $a=H_2$. The scale invariance of $\bfH(\xi)$ allows us to write $\bfH(\xi_\bfv)=\bfH(H_1+bH_2)$ where $b=\frac{v_2}{v_1}$. For convenience we shall write $\bfH(b)$ instead of $\bfH(H_1+bH_2)$ and $\bfH'(b)$ for $d\bfH_\xi(a)$. To obtain an explicit expression for $\bfH(b)$ we need to compute the total transverse scalar curvature $\bfS_\xi$ and the volume $\bfV_\xi$. To simplify the presentation, we will in the following calculations ignore any overall positive rescale that does not depend on the choice of $(v_1,v_2)$. We begin with some preliminaries. Suppose we have a quasi-regular ray in the $\bf w$-cone given by a choice of co-prime ${\bf v}= (v_1,v_2) \neq (w_1,w_2)
$. 
According to Theorem \ref{join data} above, we may assume that the transverse K\"ahler metric is admissible (for a full description of such metrics see e.g. Sections 2.3 and 5 of \cite{BoToJGA}). The volume form is then given by
$$dv_{g_\bfv} = \frac{1}{(d_N+1)!}\eta_{\bf v} \wedge (d\eta_{\bf v})^{d_N+1}. $$
By Theorem \ref{join data} and Equation (42) in \cite{BoToJGA}, we have $(d_N+1)!dv_{g_\bfv}=$
\begin{equation}\label{voleqn}
= \eta_{\bf v} \wedge \pi_\bfv^*(\frac{\gro_B}{mv_1v_2})^{d_N+1} = \frac{n^{d_N}}{(mv_1v_2)^{d_N+1}}\eta_{\bf v} \wedge \pi_\bfv^*\left((r^{-1} + \gz)^{d_N} \omega_N^{d_N} \wedge d\gz \wedge \theta \right). 
\end{equation}
By Theorem \ref{join data} and Equation (46) in \cite{BoToJGA}, the transverse scalar curvature is given by
\begin{equation}\label{Scaleqn}
Scal^T = mv_1v_2 \pi_{\bfv}^*Scal_B = mv_1v_2 \pi_{\bfv}^*\left(\frac{2 d_N s_{N_n} r }{1+r\gz} - \frac{F''(\gz)}{\gp(\gz)}\right),
\end{equation}
where $\gp(\gz) =(1 + r \gz)^{d_{N}}$,  $s_{N_n} = A/n = \frac{ A s}{l_1(w_1v_2-w_2v_1)}$, and $F$ is a smooth function satisfying
\begin{equation}
\label{positivityF}
\begin{array}{l}
F(\gz) > 0, \quad -1 < \gz <1,\\
\\
F(\pm 1) = 0,\\
\\
F'(- 1) = 2\,\gp(-1)/m_2 \quad \quad F'( 1) =-2\,\gp(1)/m_1.
\end{array}
\end{equation}
Here $2d_N A$ denotes the constant scalar curvature of $\omega_N$, so for example if $N=\bbc\bbp^{d_N}$ then $A=d_N+1$ and if $N$ is a compact Riemann surface of
genus $\calg$ then $A=2(1-\calg)$.

\begin{lemma}\label{Hwcone}
On the manifolds $M_{l_1,l_2,\bfw}$ with $b=v_2/v_1 \neq w_2/w_1$ - up to an overall positive constant rescale that does not depend on $(v_1,v_2)$ - the Einstein-Hilbert functional takes the form
\begin{equation*}\label{HE2}
\begin{array}{ccl}
\bfH(b)&=& \frac{ \left(l_1w_1^{d_N+1}b^{d_N+2}+ (l_2A-l_1w_2)w_1^{d_N}b^{d_N+1} + (l_1w_1- l_2A)w_2^{d_N}b  -l_1w_2^{d_N+1}\right)^{d_N+2}}{ (w_1 b-w_2)\left( w_1^{d_N+1}b^{d_N+2} - w_2^{d_N+1}b\right)^{d_N+1}}.\\
\end{array}
\end{equation*}
Furthermore, we have the boundary behavior 
$$\lim_{b\rightarrow 0}\bfH(b) =+\infty, \qquad \lim_{b\rightarrow +\infty}\bfH(b) =+\infty. $$
\end{lemma}

\begin{proof}
Now from Equations \eqref{voleqn} and \eqref{Scaleqn} we have
\begin{eqnarray*}
&\int_{M_{l_1,l_2,\bfw}} Scal^T dv_{g_\bfv}  \\
&=  \frac{n^{d_N}}{(mv_1v_2)^{d_N}}\int_{M_{l_1,l_2,\bfw}} \eta_{\bf v} \wedge \pi_{\bfv}^*\left(\left(\frac{2 d_N s_{N_n} r }{1+r\gz} - \frac{F''(\gz)}{\gp(\gz)}\right)(r^{-1} + \gz)^{d_N} \omega_N^{d_N} \wedge d\gz\wedge\theta\right),
\end{eqnarray*}
and
$$\int_{M_{l_1,l_2,\bfw}} dv_{g_\bfv} =  \frac{n^{d_N}}{(mv_1v_2)^{d_N+1}}\int_{M_{l_1,l_2,\bfw}} \eta_{\bf v} \wedge \pi_{\bfv}^*\left((r^{-1} + \gz)^{d_N} \omega_N^{d_N} \wedge d\gz\wedge\theta\right).$$
Applying Fubini's Theorem we have
$$\int_{M_{l_1,l_2,\bfw}} Scal^T\,dv_{g_\bfv} =  \frac{n^{d_N}}{(mv_1v_2)^{d_N}}\int_{S_n}\left( \int_{L_\xi} \eta_{\bf v}\right) \left(\frac{2 d_N s_{N_n} r }{1+r\gz} - \frac{F''(\gz)}{\gp(\gz)}\right)(r^{-1} + \gz)^{d_N} \omega_N^{d_N} \wedge d\gz\wedge\theta,$$
and
$$\int_{M_{l_1,l_2,\bfw}} dv_{g_\bfv} =  \frac{n^{d_N}}{(mv_1v_2)^{d_N+1}}\int_{S_n}\left( \int_{L_\xi} \eta_{\bf v}\right)(r^{-1} + \gz)^{d_N}  \omega_N^{d_N} \wedge d\gz\wedge\theta,$$
where $L_\xi$ is a generic leaf of the foliation $\calf_\xi$. Now by (1) of Proposition 3.4 in \cite{BoToJGA} we have that
$$\int_{L_\xi} \eta_{\bf v}= 2\pi/s = 2\pi m/{l_2}$$ and so without loss of generality we may
say that
$$
\begin{array}{ccl}
&& \int_{M_{l_1,l_2,\bfw}} Scal^T\,dv_{g_\bfv}  \\
\\
& = & (\frac{n}{mv_1v_2})^{d_N} \int_{-1}^1  \left(\frac{2 d_N s_{N_n}m r }{1+r\gz} - \frac{mF''(\gz)}{\gp(\gz)}\right)(r^{-1} + \gz)^{d_N} \,d\gz\\
\\
&=&  (\frac{n}{mv_1v_2})^{d_N} \left( \int_{-1}^1 2 d_N s_{N_n}m (r^{-1} + \gz)^{d_N-1} \,d\gz - r^{-d_N}  \int_{-1}^1 m F''(\gz)\,d\gz\right)\\
\\
&=&  (\frac{n}{mv_1v_2})^{d_N} \left( 2s_{N_n}m((r^{-1} +1)^{d_N} - (r^{-1} -1)^{d_N}) - r^{-d_N}(mF'(1)-mF'(-1))\right)\\
\\
&=& (\frac{n}{mv_1v_2})^{d_N} \left( \frac{2Am}{n}((r^{-1} +1)^{d_N} - (r^{-1} -1)^{d_N}) +2 r^{-d_N}((1+r)^{d_N}/v_1 + (1-r)^{d_N}/v_2)\right)\\
\\
&=&2^{d_N+1} (\frac{l_1}{l_2})^{d_N}\frac{\left(l_2 A v_1v_2((w_1v_2)^{d_N} - (w_2v_1)^{d_N}) +l_1(w_1 v_2-w_2v_1)(w_1^{d_N}v_2^{d_N+1} + w_2^{d_N} v_1^{d_N+1})\right)}{l_1(w_1 v_2-w_2v_1)v_1^{d_N+1}v_2^{d_N+1}}
\end{array}
$$
and
$$
\begin{array}{ccl}
&& \int_{M_{l_1,l_2,\bfw}} dv_{g_\bfv}\\
\\
& = &(\frac{n}{m})^{d_N}\frac{1}{(v_1v_2)^{d_N+1}}\int_{-1}^{1} (r^{-1} + \gz)^{d_N}\\
\\
& = & (\frac{n}{m})^{d_N}\frac{1}{(d_N+1)(v_1v_2)^{d_N+1}}\left( (r^{-1} +1)^{d_N+1} - (r^{-1}-1)^{d_N+1})\right)\\
\\
&=& 2^{d_N+1}(\frac{l_1}{l_2})^{d_N}\frac{\left((w_1v_2)^{d_N+1} - (w_2v_1)^{d_N+1} \right)}{(d_N+1)(w_1 v_2-w_2v_1)v_1^{d_N+1}v_2^{d_N+1}}.
\end{array}
$$
We set $b = \frac{v_2}{v_1} \neq \frac{w_2}{w_1}$ and  arrive at the following (where, as mentioned above, we ignore any overall positive rescale that does not depend on the choice of $(v_1,v_2)$).

\begin{equation}\label{totalScal}
\begin{array}{ccl}
&& \int_{M_{l_1,l_2,\bfw}} Scal^T\,dv_{g_\bfv} \\
\\
& = & \frac{\left(l_2 A b((w_1b)^{d_N} - w_2^{d_N}) +l_1(w_1 b-w_2)(w_1^{d_N}b^{d_N+1} + w_2^{d_N})\right)}{(w_1 b-w_2)v_1^{d_N +1}b^{d_N+1}}\\
\\
&=&  \frac{\left(l_2 A (w_1^{d_N-1} b^{d_N} + w_1^{d_N-2}w_2 b^{d_N-1} + \,\cdots \,+ w_1 w_2^{d_N-2}b^2 + w_2^{d_N-1}b) +l_1(w_1^{d_N}b^{d_N+1} + w_2^{d_N})\right)}{v_1^{d_N +1}b^{d_N+1}}\\
\end{array}
\end{equation}

\begin{equation}\label{Vol}
\begin{array}{ccl}
&& \int_{M_{l_1,l_2,\bfw}} dv_{g_\bfv}\\
\\
&= &   \frac{\left((w_1b)^{d_N+1} - w_2^{d_N+1} \right)}{(w_1 b-w_2)v_1^{d_N+2}b^{d_N+1}}\\
\\
& = &  \frac{\left(w_1^{d_N} b^{d_N} + w_1^{d_N-1}w_2 b^{d_N-1} +\, \cdots \,+ w_1 w_2^{d_N-1}b + w_2^{d_N}\right)}{v_1^{d_N+2}b^{d_N+1}}.
\end{array}
\end{equation}

\begin{rem}\label{Ageq0rem}
Even though the above calculations were done assuming that $(v_1,v_2)$ defined a quasi-regular ray, i.e. that $b \in \bbq^+$ and that $b \neq \frac{w_2}{w_1}$,  continuity and the Approximation Theorem, due to Rukimbira \cite{Ruk95}, saying that the irregular rays ($b$ non-rational) can be approximated by a sequence of quasi-regular rays, tells us that 
\eqref{totalScal} and \eqref{Vol} hold for all $b \in \bbr^+$, i.e. for all rays in the $\bfw$-cone. We also note the obvious fact that if $A \geq 0$, i.e. the constant scalar curvature of $(N,\omega_N)$ is non-negative, then the total transverse scalar curvature $ \int_{M_{l_1,l_2,\bfw}} Scal^T\,dv_{g_\bfv} >0$ for all $b\in \bbr^+$. However, when $A<0$ and $l_2$ is sufficiently large, then there may exist up to two values of $b\in \bbr^+$ where $ \int_{M_{l_1,l_2,\bfw}} Scal^T\,dv_{g_\bfv} =0$ (and in-between those two values, $ \int_{M_{l_1,l_2,\bfw}} Scal^T\,dv_{g_\bfv}$ is negative).
\end{rem}

Now we are ready to write down the Einstein-Hilbert functional for the $\bfw$-cone. 
\begin{equation}\label{HE1}
\begin{array}{ccl}
\bfH(b)&=& \frac{\left(\int_{M_{l_1,l_2,\bfw}} Scal^T\,dv_{g_\bfv}\right)^{d_N+2}}{ \left(\int_{M_{l_1,l_2,\bfw}} dv_{g_\bfv}\right)^{d_N+1}}\\
\\
&=&\frac{\left( l_2 A(w_1^{d_N-1} b^{d_N} + w_1^{d_N-2}w_2 b^{d_N-1} + \,\cdots \,+ w_1 w_2^{d_N-2}b^2 + w_2^{d_N-1}b) + l_1(w_1^{d_N} b^{d_N+1} + w_2^{d_N})\right)^{d_N+2}}{\left(w_1^{d_N} b^{d_N+1} + w_1^{d_N-1}w_2 b^{d_N} +\, \cdots \,+ w_1 w_2^{d_N-1}b^2 + w_2^{d_N}b\right)^{d_N+1}}
\end{array}
\end{equation}
which reduces to the given form when $b\neq w_2/w_1$.
\end{proof}

\subsection{Admissible CSC constructions}\label{sasadmsec}
Before we consider the critical points of $\bfH(b)$, we recall some observations from \cite{BoToJGA}  concerning the existence of admissible CSC metrics on $M_{l_1,l_2,\bfw}=M\star_{l_1,l_2}S^3_\bfw$. We refer to sections 5.1 and 6.2 (in particular equation (67) and the comments surrounding it) of  \cite{BoToJGA},  for the justification of the statement below.

\begin{proposition}\label{CSCexistence}
Consider a ray in the $\bfw$-cone determined by a choice of $b>0$. Then the Sasakian structures of the ray has admissible CSC metrics (up to isotopy) if and only if $f_{CSC}(b)=0$, where
\begin{equation}\label{fCSCeqn}
f_{CSC}(b) = \frac{-f(b)}{(w_1 b - w_2)^3}
\end{equation} 
and $f(b)$ is a polynomial given as follows:
\begin{equation}\label{functionf}
\begin{array}{ccl}
f(b) & = &-(d_N + 1 ) l_1w_1^{2d_N +3} b^{2 d_N+4}\\
\\
& + & w_1^{2(d_N+1)} b^{2 d_N+3}( A l_2 + l_1 (d_N + 1)  w_2)\\
\\
& - & w_1^{d_N+2}  w_2^{d_N} b^{d_N + 3}  ((d_N+1)  (A ( d_N+1) l_2 - l_1 (( d_N+1) w_1 + ( d_N+2) w_2)))\\
\\
& + & w_1^{d_N+1}  w_2^{d_N+1}  b^{d_N + 2}  (2 A d_N (d_N+2) l_2 - (d_N+1)(2d_N+3) l_1 (w_1 + w_2))\\
\\
& - &  w_1^{d_N}  w_2^{d_N+2}  b^{d_N + 1}(d_N+1) (A ( d_N+1) l_2 - l_1 ((d_N+2) w_1 + ( d_N+1) w_2))\\
\\
& + & w_2^{2 (d_N + 1)}( b (A l_2 + l_1 (d_N + 1) w_1))\\
\\
& - & ( d_N + 1 ) l_1 w_2^{2 d_N + 3} .
\end{array}
\end{equation}
This polynomial has a root of order three at $b=w_2/w_1$ when $w_1>w_2$ and order at least four when $w_1=w_2=1$ (where the case of $b=w_2/w_1=1$ gives a product transverse CSC structure).
Thus $f_{CSC}(b)$ is a polynomial of order $2d_N+1$ with positive roots corresponding to the rays in the $\bfw$-cone that admit admissible CSC metrics.
\end{proposition}

\subsection{Derivative of $\bfH(b)$}

Using Lemma \ref{HE2}, we calculate the  derivative of $\bfH(b)$:
\begin{equation}\label{dervH1}
\begin{array}{ccl}
&& \bfH'(b)\\
\\
&=& 
  \frac{
\left( l_2 Ab (w_1^{d_N} b^{d_N} - w_2^{d_N} )+l_1(w_1 b - w_2)(w_1^{d_N } b^{d_N+1}   + w_2^{d_N} )\right)^{d_N+1}(w_1 b - w_2) f_{CSC}(b)}
{b^{d_N+2}(w_1^{d_N+1} b^{d_N+1} - w_2^{d_N+1} )^{d_N+2}},
\end{array}
\end{equation}
We may rewrite \eqref{dervH1} as follows
\small{
\begin{equation}\label{dervH2}
\begin{array}{ccl}
&& \bfH'(b)\\
\\
&=& 
 \frac{
\left( l_2 A (w_1^{d_N-1} b^{d_N} + w_1^{d_N-2}w_2 b^{d_N-1} + \,\cdots \,+ w_1 w_2^{d_N-2}b^2 + w_2^{d_N-1}b )+l_1(w_1^{d_N } b^{d_N+1}   + w_2^{d_N} )\right)^{d_N+1}f_{CSC}(b)}
{b^{d_N+2}(w_1^{d_N} b^{d_N} + w_1^{d_N-1}w_2 b^{d_N-1} +\, \cdots \,+ w_1 w_2^{d_N-1}b + w_2^{d_N} )^{d_N+2}}.
\end{array}
\end{equation}}

This tells us that for $A\geq 0$, the critical points of $\bfH(b)$ correspond exactly to the ray(s) with admissible cscS metrics. In this case notice also that the sign of $\bfH'(b)$ equals the sign of 
$f_{CSC}(b)$. 
For $A<0$, $\bfH(b)$ may have critical points that do not correspond to admissible cscS rays. See for instance Example \ref{criticalptsnoncsc} below.

An important observation is now that if we compare \eqref{dervH2} with \eqref{totalScal} and \eqref{Vol} we see that

$$
\begin{array}{ccl}
\bfH'(b) &=& \frac{\left(v_1^{d_N +1}b^{d_N+1} \int_{M_{l_1,l_2,\bfw}} Scal^T\,dv_{g_\bfv} \right)^{d_N+1}f_{CSC}(b)}
{b^{d_N+2}\left(  v_1^{d_N+2}b^{d_N+1}  \int_{M_{l_1,l_2,\bfw}} dv_{g_\bfv}   \right)^{d_N+2}}  \\
\\
&=& \frac{\left( \int_{M_{l_1,l_2,\bfw}} Scal^T\,dv_{g_\bfv} \right)^{d_N+1}f_{CSC}(b)}
{(bv_1)^{2d_N+3}\left( \int_{M_{l_1,l_2,\bfw}} dv_{g_\bfv}   \right)^{d_N+2}}. 
\end{array}
$$
Comparing this to Lemma \ref{lemmaCRIT} and using the fact that $\bfH(b)$ is a rational function with only isolated zeroes and that $f_{CSC}$ as well as the Sasaki-Futaki invariant varies smoothly in the Sasaki-cone, we conclude that, up to a positive multiple ,
$f_{CSC}(b)$ represents the value of the Sasaki-Futaki invariant $\bfF_\xi$ at $\xi$ given by $(v_1,v_2)$ in a direction tranversal to the rays. Thus Proposition \ref{CSCexistence} tells us that in this case the vanishing of the Sasaki-Futaki invariant implies the existence of cscS metrics for the given ray and thus in cases where the $\bfw$-cone is the entire Sasaki cone we have that the
existence of cscS metrics is equivalent to the vanishing of  $f_{CSC}(b)$ and hence with the vanishing of $\bfF_\xi$. We have arrived at

\begin{lemma}\label{explicitSF}
On the manifolds $M_{l_1,l_2,\bfw}$ - up to an overall positive constant rescale - the Sasaki-Futaki invariant takes the form
$$\bfF_\xi(\Phi(H_2))=\frac{1}{v_2^{2n+1}\bfV_\xi}f_{CSC}(b)$$
where $f_{CSC}(b)$ is given by Equations \eqref{fCSCeqn} and \eqref{functionf}.
\end{lemma}

Note that this theorem together with Corollary 1 of \cite{CollinsSz} proves Theorem \ref{stabcsc} of the Introduction. \hfill $\Box$

As a bonus Lemma \ref{explicitSF} and Proposition 5.2 of \cite{BGS} gives
\begin{corollary}\label{wfull}
If on $M_{l_1,l_2,\bfw}$ with $\xi\in \gt^+_\bfw$ the Sasaki-Futaki invariant $\bfF_\xi$ vanishes on the subalgebra $\gt_\bfw\otimes\bbc$, then it vanishes identically.
\end{corollary}

\subsection{Some Examples}\label{exsec}
In this section we give some examples which illustrate our methods when applied to the $S^3_\bfw$-join construction. In all the examples we give here the $\bfw$-cone is the full Sasaki cone, but we could easily give examples where the full Sasaki cone is bigger with our statements only applying to the $\bfw$-subcone.

\begin{example}\label{tripleCSC}
Let us now assume that $N= \bbc\bbp^2 \# k \overline{\bbc\bbp}^2$ for $4\leq k \leq 8$ ($\bbc\bbp^2$ blown up at $k$ generic points). Then the $\bfw$-cone is the full Sasaki cone. Moreover, $N$ has a KE metric with primitive K\"ahler form $\omega_N$ and, by a theorem of Kobayashi and Ochiai \cite{KoOc73}, $A=1$. We know that in the Sasaki-Einstein case, that is, $c_1(D)=0$, we must have $l_1=1$ and $l_2=w_1+w_2$, and that the functional $\bfH(\xi)$ is (transversally) convex. It is also probably convex in some cases with $c_1(D)\neq 0$ as well. However, if we set $l_1=1$, $w_1=3$, $w_2=2$, and $l_2 = 29$, for example, we have
{\small $$
\begin{array}{ccl}
\bfH(b) & =  & \frac{\left(4+58 b+87 b^2+9 b^3\right)^4}{b^3 \left(4+6 b+9 b^2\right)^3} \\
\\
\bfH'(b) & = & \frac{\left(4+58 b+87 b^2+9 b^3\right)^3 \left(-48+88 b+720 b^2-1242 b^3-459 b^4+243 b^5\right)}{b^4 \left(4+6 b+9 b^2\right)^4}\\
\\
f_{CSC}(b) & = &  243 b^5- 459 b^4- 1242 b^3+ 720 b^2+ 88 b-48\\
\end{array}
$$}
Since $f_{CSC}(0) = -48 <0$, $f_{CSC}(1/3) = 32/3>0$, $f(2/3) = -96 <0$,  and $\lim_{b \rightarrow +\infty}f_{CSC}(b) = + \infty$
we know that $f_{CSC}$ has at least three distinct positive roots and applying Descarte's rule of signs to $f_{CSC}$ we see it has at most 
three positive roots ... so it has \underline
{exactly} three distinct positive roots. Looking at $\bfH'(b)$ above (or simply using that its sign follows the sign of $f_{CSC}(b)$ we observe that the root in the middle is a relative maximum of $f_{CSC}$ while the others are relative minima. At the maximum, obviously, $\bfH''(b)<0$, while at the minima $\bfH''(b)>0$. So $\bfH(\xi)$ fails to be convex in this case. This lack of global convexity on manifolds of the form $M_{l_1,l_2,\bfw}$ occurs in infinitely many cases; in particular, it occurs on $M_{1,l_2,(3,2)}$ (which is a smooth manifold when $\gcd(l_2,6)=1$) for all $l_2\geq 29$. 

\end{example}

\begin{example} \label{criticalptsnoncsc} 
Assume now that $d_N=1$, and $N=\grS_2$ is a genus two compact Riemann surface with $A=-2$ , $l_1=1$, $w_1=3$, $w_2=2$,  and $l_2$ is any positive integer such that $\gcd(l_2,6)=1$.
In this case we have that
$$
\begin{array}{ccl}
\bfH(b) & =  & \frac{(3b^2-2l_2 b +2)^3}{b^2(3 b + 2)^2}\\
\\
\bfH'(b) & = & \frac{2 (3b^2 - 2l_2 b +2)^2 (9 b^3 + (3 l_2+12) b^2  - (12+l_2) b - 4)}{b^3(3b +2)^3}\\
\\
f_{CSC}(b) & = & 2 (9 b^3 + (3 l_2 +12) b^2  - (12+l_2) b - 4)\\
\end{array}
$$
Further, one may verify that the \emph{extremal function}, $F_{ext}(\gz)$, satisfying the (admissible) extremal ODE (arising from setting $Scal_B$ of \eqref{Scaleqn} equal to an affine function of $\gz$)  and the endpoint conditions of Equation \eqref{positivityF}  is - up to a positive rescale -
equal to $(1-\gz^2)g(\gz)$, where 
$$
\begin{array}{ccl}
&& g(\gz)\\
\\
 & = &( 3 b-2) (9 b^3 + 12 b^2 - 12 b  -4 ) \gz^2 + 2 (3 b^2-2) ( 9 b^2+ 24 b+4 ) \gz \\
 \\
& + & 27 b^4 + 126 b^3 + 120 b^2+ 84 b+ 8   \\
\\
& - &  l_2 b (3 b-2)^2 (1 - \gz^2).
\end{array}$$
For a given positive $b \neq 2/3$, the corresponding ray has an \emph{admissible} extremal metric if and only if $F_{ext}(\gz)$, hence $g(\gz)$, is positive for all $-1<\gz <1$.

Now let us look at what happens for a small value of $l_2$ and a large value of $l_2$:
\begin{itemize}
\item For $l_2=1$, the manifold $M_{1,1,(3,2)}$ is the nontrivial $S^3$-bundle over $\grS_2$. We have that the solution ($b>0$) to $f_{CSC}(b)=0$ (i.e. the cscS solution) is approximately $b= 0.835$ and in this case, this is the only critical point of $\bfH(b)$ for $b>0$. It is in fact a minimum.   

\bigskip

\item For $l_2=101$, the manifold $M_{1,101,(3,2)}$ is a nontrivial lens space bundle over $\grS_2$. In this case we have three solutions. First there is the solution to $f_{CSC}(b)=0$ which is approximately $b= 0.685$ and gives a cscS Sasaki metric. At this value of $b$ we still have a local minimum of $\bfH(b)$ (and here $\bfH(b)$ is negative). However, $\bfH(b)$ has two additional critical points which are not local extrema, but inflection points, namely; $b= \frac{2}{101+\sqrt{10195}} \approx 0.099$ and $b= \frac{ \left(101+\sqrt{10195}\right)} {3}\approx 67.3$. These are examples of critical points of $\bfH(b)$ where the Sasaki-Futaki invariant does not vanish. One may check that for either of these values
$g(\gz)$ fails to be positive for all $-1<\gz <1$ and thus the corresponding rays do not allow admissible extremal metrics. These metrics are K-unstable by Theorem \ref{theoSTABcrit}. In this case a computer analysis indicates that there are numbers $b_1\approx 0.295$ and $b_2\approx 1.455$ such that
for $b_1<b<b_2$ we have admissible extremal Sasaki metrics; whereas, for $b$ outside this interval (either side), positivity of $g$ fails and there are no admissible extremal metrics in these two components.
\end{itemize}
\end{example}

\begin{example} \label{bothscal0csc} 
Here we take $N=\grS_\calg$ to be a compact Riemann surface of genus $\calg>1$, so $d_N=1$. Take $l_2A= - 2l_1\sqrt{w_1w_2}$ where $l_2$ is any positive integer such that $\gcd(l_2,w_1w_2)=1$. Since $A=2l_2(1-\calg)$, we see that $\calg=1+\frac{l_1}{l_2}\sqrt{w_1w_2}$.
Now since $l_2$ is relatively prime to $l_1w_1w_2$ we must have $l_2=1$ in which case $M_{l_1,l_2,\bfw}$ is an $S^3$-bundle over $\grS_\calg$. Moreover, $\calg$ must take the form $\calg=1+l_1p_1^{r_1}p_2^{r_2}$ where $p_1,p_2$ are distinct primes, $r_1,r_2\in\bbz^+$ and $\bfw=(p_1^{2r_1},p_2^{2r_2})$. Note that $M_{l_1,l_2,\bfw}$ is the trivial bundle $\grS_\calg\times S^3$ if $l_1(w_1+w_2)$ is even, and the non-trivial bundle $\grS_\calg\tilde{\times} S^3$ when $l_1(w_1+w_2)$ is odd.

Then from Lemma \ref{Hwcone} we find
$$\bfH(b) =\frac{(l_1 w_1 b^2 + l_2 A b + l_1 w_2)^3 }{(b^2(w_1 b + w_2)^2)} $$
and 
$$f_{CSC}(b) = 2l_1 w_1^2 b^3 - w_1(l_2 A - 4 l_1 w_2)b^2 + w_2(l_2 A - 4 l_1 w_1) b -2 l_1 w_2^2$$
One can check that $b=\sqrt{w_2/w_1}$ is a root of both $f_{CSC}(b)$ and  $l_1 w_1 b^2 + l_2 A b + l_1 w_2$.
In fact,   
$$l_1 w_1 b^2 + l_2 A b + l_1 w_2 = l_1w_1\left(b-\sqrt{\frac{w_2}{w_1}}\right)^2,$$
so $\bfH(b)$ has a six-tuple root at $b=\sqrt{w_2/w_1}$. Thus, we recover the cscS metrics $g$ on the manifolds $\grS_\calg\times S^3$ and $\grS_\calg\tilde{\times} S^3$ with constant transverse scalar curvature $s^T_g=0$ described at the end of Section 5.5 in \cite{BoToIMRN}.

\end{example}

\bibliographystyle{abbrv}

\begin{thebibliography}{DDDD}


 \bibitem{aubin}
   {\scshape T. Aubin}
     \emph{Equations de type Monge-Amp\`ere sur les vari\'et\'es k\"ahl\'eriennes compactes}, C. R. Acad. Sci. Paris {\bf 283} (1976), 119--121.


 \bibitem{besse}
  {\scshape A. Besse}
    \emph{Einstein manifolds}, Classics in mathematics, Springer-Verlag Berlin, 1987.
    

 \bibitem{contactNOTE}
  {\scshape C. P. Boyer,  K. Galicki}
    \emph{A note on toric contact geometry}, J. Geom. Phys. {\bf 35} (2000), 288--298.    

 \bibitem{BG:book}
  {\scshape C. P. Boyer,  K. Galicki}
    \emph{Sasakian geometry}, Oxford Mathematical Monographs. Oxford University Press, Oxford, 2008.

  \bibitem{BGS}
   {\scshape C. P. Boyer,  K. Galicki, S. R. Simanca}
     \emph{Canonical Sasakian metrics}, Commun. Math. Phys. {\bf 279} (2008), 705--733.

\bibitem{BoToIMRN}
 {\scshape C. P. Boyer, C. W. T{\o}nnesen-Friedman}
   \emph{Extremal Sasakian geometry on $S^3$-bundles over Riemann surfaces}, Int. Math. Res. Not. 2014 no. 20, 5510-5562.

\bibitem{BoToJGA}
 {\scshape C. P. Boyer, C. W. T{\o}nnesen-Friedman}
   \emph{The {S}asaki join, {H}amiltonian 2-forms, and constant scalar curvature}, J. of Geom. Anal., to appear (2015).

\bibitem{BoToNY}
  {\scshape C. P. Boyer, C. W. T{\o}nnesen-Friedman}
    \emph{On the Topology of some Sasaki-Einstein manifolds}, New York J. Math. {\bf 21} (2015), 57�72. 
    
\bibitem{BoToPaul}
  {\scshape C. P. Boyer, C. W. T{\o}nnesen-Friedman}
    \emph{The Sasaki join and admissible K\"ahler constructions}, J. Geom. Phys. 91 (2015), 29�39.  
    
\bibitem{BoToinfcon}
  {\scshape C. P. Boyer, C. W. T{\o}nnesen-Friedman}
    \emph{Simply connected manifolds with infinitely many toric contact structures and constant scalar curvature Sasaki metrics}, mathDG:arXiv.org/1404.3999.

 \bibitem{calabi}
   {\scshape E. Calabi}
   \emph{Extremal K\"ahler metrics. II.}, Differential geometry and complex analysis, I. Chavel, H. M. Farkas Eds., 95--114, Springer, Berlin, 1985.
   
  \bibitem{CDSI}
    {\scshape X.X. Chen, S. Donaldson, S. Sun}
       \emph{K\"ahler-Einstein metrics on Fano manifolds. I: Approximation of metrics with cone singularities}, J. Amer. Math. Soc. 28 (2015), 183-197.
       
 \bibitem{CDSII}
    {\scshape X.X. Chen, S. Donaldson, S. Sun}
       \emph{K\"ahler-Einstein metrics on Fano manifolds. II: Limits with cone angle less than $2\pi$}, J. Amer. Math. Soc. 28 (2015), 199- 234.
       
       
 \bibitem{CDSIII}
    {\scshape X.X. Chen, S. Donaldson, S. Sun}
       \emph{K\"ahler-Einstein metrics on Fano manifolds. III: Limits as cone angle approaches $2\pi$ and completion of the main proof}, J. Amer. Math. Soc. 28 (2015), 235-278.           
   
 \bibitem{CollinsSz}
  {\scshape T. Collins, G.Sz\'ekelyhidi}
  \emph{K-Semistability for irregular Sasakian manifolds}, to appear in J. Differential Geom. mathDG:arXiv.org/1204.2230.
  
 
\bibitem{don:estimate}
  {\scshape S. K. Donaldson}
    \emph{Interior estimates for solutions of Abreu's equation}, Collect. Math. {\bf 56} (2005), 103--142.

 \bibitem{don:extMcond}
  {\scshape S. K. Donaldson}
    \emph{Extremal metrics on toric surfaces: a continuity method}, J. Differential Geom. {\bf 79} (2008), 389--432.

 \bibitem{don:csc}
  {\scshape S. K. Donaldson}
      \emph{Constant scalar curvature metrics on toric surfaces}, Geom. Funct. Anal. {\bf 19} (2009), 83--136.
      
\bibitem{futaki}
  {\scshape A. Futaki}
      \emph{An obstruction to the existence of Einstein K\"ahler metrics}, Invent. Math. {\bf 73} (1983), no. 3, 437--443.
  
 \bibitem{FutakiOnoWang}
  {\scshape A. Futaki, H. Ono, G. Wang}
	\emph{Transverse K\"ahler geometry of Sasaki manifolds and toric Sasaki--Einstein manifolds}, J. Differential Geom. {\bf 83} (3) (2009) 585--635.

\bibitem{Gaubook}
       {\scshape Gauduchon, Paul},
      \emph{Calabi's extremal {K}\"ahler metrics: an elementary introduction}, Notes: preliminary version, 2010.
     
\bibitem{KoOc73} 
 {\scshape S. Kobayashi, T. Ochiai} \emph{Characterizations of complex
projective spaces and hyperquadrics}, J. Math. Kyoto Univ. 13 (1973),  31--47.
    
\bibitem{moi2}  
{\scshape E. Legendre}
\emph{ Existence and non-uniqueness of constant scalar curvature toric Sasaki metrics}, Compositio Math. {\bf 147} (2011),
  1613--1634.
  
   \bibitem{L:contactToric}
  {\scshape E. Lerman}
    \emph{Toric contact manifolds}, J. Symplectic Geom. {\bf 1} (2003), 785--828.
    
 \bibitem{reebMETRIC}
  {\scshape D. Martelli,  J. Sparks, S.-T. Yau}
    \emph{The geometric dual of {\it a}-- maximisation for toric Sasaki-Einstein manifolds}, Comm. Math. Phys.  {\bf 268}  (2006), 39--65.

 \bibitem{MSYvolume}
  {\scshape D. Martelli,  J. Sparks, S.-T. Yau}
    \emph{Sasaki-Einstein manifolds and volume minimisation}, Comm. Math. Phys. {\bf 280}  (2008), 611--673.

    
     \bibitem{matsushima}
  {\scshape Y. Matsushima}
  \emph{Sur la structure du groupe d'hom\'eomorphismes analytiques d'une certaine vari\'et\'e k\"ahl\'erienne}, Nagoya Math. J. {\bf 11} (1957), 145--150.
  
\bibitem{RoThJDG}
   {\scshape J. Ross, R. Thomas}
    \emph{Weighted Projective Embeddings, Stability of Orbifolds, and Constant Scalar Curvature K\"ahler Metrics}, J. Diff. Geom. {\bf 88} (2011), 109-159. 
  
  \bibitem{Ruk95}
  {\scshape P. Rukimbira}
  \emph{Chern-Hamilton's conjecture and $K$-contactness}, Houston J. Math. {\bf 21} No 4 (1995), 709--718.

 \bibitem{tian:conj}
  {\scshape G. Tian}
    \emph{On Calabi's conjecture for complex surfaces with positive first Chern class}, Invent. Math. {\bf 101} (1990), 101--172.
    
 \bibitem{T}
  {\scshape G. Tian}
    \emph{ K\"ahler--Einstein metrics with positive scalar curvature}, Inventiones Math. {\bf 130} No 1 (1997), 1--37.
    
 \bibitem{T15}
  {\scshape G. Tian}
    \emph{K\"ahler-Einstein metrics on Fano manifolds},  Japan. J. Math. 10 (2015), no. 1, 1-41.   
 
 \bibitem{T15b}
  {\scshape G. Tian}
    \emph{K-Stability and K\"ahler-Einstein Metrics},  Communications on Pure and Applied Math. 68 (7) (2015), 1085-1156. 
    
    
\bibitem{TivC}
  {\scshape C. Tipler, C. van Coevering}
    \emph{Deformations of Constant Scalar Curvature Sasakian Metrics and K-Stability}, Int. Math. Res. Not., to appear (2015).      
    
    
 \bibitem{Y}
  {\scshape S.T. Yau}
    \emph{On the Ricci curvature of a compact K\"ahler manifold and the complex Monge-Amp\`ere equation I}, Comm. Pure Appl. Math., {\bf 31} (1978), 339-411.
    
    
 \bibitem{yau}
  {\scshape S.-T. Yau}
    \emph{Open problems in geometry}. Differential geometry: Partial differential equations on manifolds (Los Angeles, CA, 1990), 1--28, Proc. Sympos. Pure Math., 54, Part 1, Amer. Math. Soc., Providence, RI, 1993.
    
    

\end{thebibliography}

\end{document}